\documentclass[aos,preprint,authoryear]{imsart}

\RequirePackage[OT1]{fontenc}
\RequirePackage{amsthm,amsmath,amssymb}
\RequirePackage{natbib}
\RequirePackage[colorlinks,citecolor=blue,urlcolor=blue]{hyperref}

\usepackage{graphicx}
\usepackage{xcolor}
\usepackage{subfig}
\usepackage{xcolor}
\usepackage{subfig}
\usepackage{comment}
\usepackage{pdfpages}

\arxiv{arXiv:1502.02069}
\startlocaldefs

\def\E{\mathbb{E}}
\def\I{\mathbb{I}}
\def\P{\mathbb{P}}
\def\R{\mathbb{R}}
\def\S{\mathcal{S}}

\def\sV{\mathcal{V}}
\def\sN{\mathcal{N}}

\def\1{\boldsymbol{1}}

\numberwithin{equation}{section}
\theoremstyle{plain}
\newtheorem{theorem}{Theorem}[section]

\newtheorem{assumption}[theorem]{Assumption}
\newtheorem{lemma}[theorem]{Lemma}
\newtheorem{corollary}[theorem]{Corollary}

\theoremstyle{remark}
\newtheorem{remark}[theorem]{Remark}

\endlocaldefs

\allowdisplaybreaks

\begin{document}

\begin{frontmatter}
\title{Likelihood-based model selection for stochastic block models}
\runtitle{Likelihood-based model selection for SBMs}

\begin{aug}
\author{\fnms{Y.X. Rachel} \snm{Wang}\thanksref{m1}\ead[label=e1]{ryxwang@stanford.edu}}
\and
\author{\fnms{Peter J.} \snm{Bickel}\thanksref{m2}\ead[label=e2]{bickel@stat.berkeley.edu}}

\runauthor{Y.X.R. Wang and P.J. Bickel}

\affiliation{Department of Statistics, Stanford University\thanksmark{m1} \\ Department of Statistics, University of California, Berkeley\thanksmark{m2}}

\address{Y.X.R. Wang\\
Department of Statistics\\
Sequoia Hall\\
390 Serra Mall\\
Stanford University\\
Stanford, CA 94305\\
\printead{e1}
}

\address{P.J. Bickel\\
Department of Statistics\\
University of California, Berkeley\\
367 Evans Hall\\
Berkeley, California, 94720\\
\printead{e2}
}

\end{aug}

\begin{abstract}
The stochastic block model (SBM) provides a popular framework for modeling community structures in networks. However, more attention has been devoted to problems concerning estimating the latent node labels and the model parameters than the issue of choosing the number of blocks. We consider an approach based on the log likelihood ratio statistic and analyze its asymptotic properties under model misspecification. We show the limiting distribution of the statistic in the case of underfitting is normal and obtain its convergence rate in the case of overfitting. These conclusions remain valid when the average degree grows at a polylog rate. The results enable us to derive the correct order of the penalty term for model complexity and arrive at a likelihood-based model selection criterion that is asymptotically consistent. Our analysis can also be extended to a degree-corrected block model (DCSBM). In practice, the likelihood function can be estimated using more computationally efficient variational methods or consistent label estimation algorithms, allowing the criterion to be applied to large networks.
\end{abstract}



\end{frontmatter}

\section{Introduction}
Network modeling has attracted increasing research attention in the past few decades as the amount of data on complex systems accumulates at an unprecedented rate. Many complex systems in science and nature consist of interacting individual components which can be represented as nodes with connecting edges in a network. Network
modeling has found numerous applications in studying friendship networks in sociology, Internet traffic in information
technology, predator-prey interactions in ecology, and protein-protein interactions and gene regulatory mechanisms in molecular biology.

One prominent feature of many of these networks is the presence of communities, where groups of nodes exhibit high internal connectivity. Communities provide a natural division of the network into subunits with certain traits. In social networks, they often arise based on people's common interests and geographic locations. The World Wide Web forms communities or hubs based on the content of the web pages. In gene networks, communities correspond to genes with related functional groupings, many of which can act in the same biological pathway. Numerous heuristic algorithms have been proposed for detecting communities. However, a generative model is needed to study the problem from a theoretical perspective. 

The stochastic block model (SBM), proposed by \citet{Holland:1983} in social science, is one of the simplest random graph models incorporating community structures. It assigns each node a latent discrete block variable and the connectivity levels between nodes are determined by their block memberships. In practice, this model sometimes oversimplifies the structures of real networks and other variants have been proposed, including the degree-corrected SBM (DCSBM) \citep{KarrerNewman:2010} relaxing the within-block degree homogeneity constraint and overlapping SBM \citep{Airoldi:2008} allowing a node to be in multiple blocks. These models have been applied to model real networks in social science and biology \citep{BickelChen:2009, Daudin:2008, Airoldi:2008, KarrerNewman:2010}. 

Much research effort has been devoted to the problems of estimating the latent block memberships and model parameters of a SBM, including modularity \citep{newman:2006} and likelihood maximization \citep{BickelChen:2009, Amini:2013}, variational methods \citep{Daudin:2008, latouche:2012}, spectral clustering \citep{Rohe:2011, fishkind:2013}, belief propagation \citep{decelle:2011} to name but a few. The asymptotic properties of some of these methods have also been studied \citep{BickelChen:2009, Rohe:2011, celisse:2012, bickel:2013}. However, these methods require knowing (or knowing at least a suitable range for) $K$, the number of blocks, a priori. Less attention has been paid to the problem of selecting $K$. 

For general networks this problem corresponds to the issue of determining the number of communities, which remains a challenging open problem. Recursive approaches have been adopted to extract \citep{zhao:2011} or divide \citep{BickelSarkar:2013} one community sequentially, while using optimization strategies or hypothesis testing to decide whether the process should be stopped at one stage. A more general sequential test for comparing a fitted SBM against alternative models with finer structures is proposed in \citet{lei:2014}. Conceptually these approaches are more appealing for networks with a hierarchical structure. In other cases, it would be more desirable to be able to compare different community numbers directly. A few likelihood-based model selection criteria have been proposed \citep{Daudin:2008, latouche:2012, saldana:2014}. From an information-theoretic perspective, \citet{peixoto:2013} proposed a criterion based on minimum length description. These approaches circumvent the difficulty of  analyzing the likelihood directly by using variational approximations or assuming the node labels are fixed and using plug-in estimates obtained from other inference algorithms. Furthermore, the asymptotic studies of these criteria examining their large-sample performance remain incomplete.  Empirically, a network cross-validation method has been investigated in \citet{chen:2014}. More recently, \cite{le:2015} proposed a method based on analyzing the spectral properties of graph operators, including the non-backtracking matrix and the Bethe Hessian matrix. 

In this paper, we directly address the challenges involved in analyzing the asymptotic distribution of the maximum log likelihood function under model misspecification. We show the log likelihood ratio statistic is asymptotically normal in the case of underfitting. Although obtaining an explicit asymptotic distribution of the statistic in the case of overfitting is much more challenging, we have still derived its order of convergence and subsequently shown these two cases of misspecification can be separated with probability tending to one. We thus propose a model selection criterion taking the form of a penalized likelihood and show it is asymptotically consistent in the regime where network average degree grows at a polylog rate. In Section \ref{sec_results}, we first derive our main results under the regular SBM assumptions and then outline how the arguments can be extended to a DCSBM. Computationally the likelihood can be approximated with variational algorithms or consistent label estimation algorithms without affecting the asymptotic consistency of the criterion. We demonstrate the effectiveness of our method by comparing its performance with other model selection approaches on simulated and real networks in Sections \ref{sec_sim} and \ref{sec_real}.

\section{Results}
\label{sec_results}
\subsection{Preliminaries}
A SBM with $K$ blocks on $n$ nodes is defined as follows. A vector of latent labels $Z=(Z_1, \dots, Z_n)$ is generated with $Z_i$ taking integer values from $[K] = \{1, \dots, K\}$ governed by a multinomial distribution with parameters $\pi=(\pi_{1},\pi_{2},\ldots,\pi_{K})$. Given $Z_{i}=a,$ $Z_{j}=b$, an adjacency matrix $A$ is generated with 
\[
A_{i,j}|(Z_{i}=a,Z_{j}=b)\sim\text{Bernoulli}(H_{a,b}),\quad\quad i\neq j.
\]
We consider a symmetric $A$ with zero diagonal entries corresponding to an undirected graph, although our arguments generalize easily to directed graphs. $H$ is a $K\times K$ symmetric matrix describing the connectivities within and between blocks. We denote the model parameters $\theta=(\pi, H)$ and let $\Theta_{K}$ be the parameter space of a $K$-block model,
\begin{align*}
\Theta_{K}= & \{ \theta \mid \pi \in (0,1)^K, \sum_{a=1}^{K} \pi_a=1, H\in(0,1)^{K\times K} \}. 
\label{eq:Theta_K}
\end{align*} 

Throughout the paper, $\theta^*=(\pi^*, H^*)$ will denote the true generative parameter giving rise to an observed $A$. We will further parametrize $H^*$ by $H^*=\rho_n S^*$, where the degree density $\rho_n$ may be $\Omega(1)$ or going to zero at a rate $n\rho_n/\log n \to \infty$. We assume $\theta^*\in\Theta_K$ and $H^*$ has no identical columns, meaning the underlying model has $K$ blocks and it is identifiable in the sense that it cannot be further collapsed to a smaller model. $z=(z_1, \dots, z_n)\in [K']^n$ represents another set of labels under a $K'$-block model with $K'$ not necessarily equaling $K$. $g(A; \theta)$ is the likelihood function describing the distribution of $A$ with parameter $\theta\in\Theta_{K'}$ and can be written as the sum of the complete likelihood function $f(z, A; \theta)$ associated with the labels $z\in [K']^n$:
\begin{equation}
g(A; \theta)  = \sum_{z\in [K']^n} f(z, A; \theta),
\label{eq_g_sum}
\end{equation}
where $f(z, A; \theta)$ takes the form
\begin{align*}
f(z, A; \theta) 
 & =  \left( \prod_{a=1}^{K'} \pi_a^{n_a(z)} \right) \left( \prod_{a=1}^{K'} \prod_{b=1}^{K'} H_{a,b}^{O_{a,b}(z)}(1-H_{a,b})^{n_{a,b}(z) - O_{a,b}(z)}   \right)^{1/2}
\end{align*}
with count statistics
\begin{align*}
n_a(z) & = \sum_{i=1}^{n}\I(z_{i}=a), 	\quad n_{a,b}(z)  =  \sum_{i=1}^{n}\sum_{j\neq i}\I(z_{i}=a,z_{j}=b)	\\
O_{a,b}(z) & =  \sum_{i=1}^{n}\sum_{j\neq i}\I(z_{i}=a,z_{j}=b)A_{i,j}.	\\
\end{align*}
$g$ and $f$ are invariant with respect to a permutation on the block labels, $\tau: [K'] \to [K']$, and its corresponding permutations on the node labels $z$ and the parameters $\theta$. Furthermore, let $R(z)$ be the $K' \times K$ confusion matrix whose $(k,a)$-th entry is 
\begin{equation}
R_{k,a}(z, Z) = n^{-1}\sum_{i=1}^{n} \I (z_i = k, Z_i = a).
\end{equation}

We take a likelihood-based approach toward model selection and first investigate whether different model choices can be separated using the log likelihood ratio
\begin{equation}
L_{K, K'} = \log \frac{\sup_{\theta\in\Theta_{K'}}g(A;\theta)}{\sup_{\theta\in{\Theta_{K}}}g(A;\theta)}.
\label{eq_llhr}
\end{equation}
Here the comparison is made between the correct $K$-block model and fitting a misspecified $K'$-block model.  

In the following sections, we analyze the asymptotic distribution of $L_{K, K'}$ for $K'\neq K$. The main focus of analysis lies in handling the sum in \eqref{eq_g_sum} which contains an exponential number of terms. It has been shown in \cite{bickel:2013} that when $\theta\in\Theta_K$, $\sup_{\theta\in\Theta_K} g(A; \theta)$ is essentially equivalent to maximizing the complete likelihood corresponding to the correct labels $Z$, $\sup_{\theta\in\Theta_K} f(Z, A; \theta)$. In the next section, we handle the case of underfitting and derive the asymptotic distribution of $L_{K, K'}$.

\subsection{Underfitting}
\label{sec_under}
We start by considering $K'=K-1$. Intuitively, a $(K-1)$-block model can be obtained by merging blocks in a $K$-block model. More specifically, given the correct labels $Z\in[K]^n$ and the corresponding block proportions $p=(p_1, \dots, p_{K})$, $p_a= n_a(Z)/n$, we define a merging operation $U_{a,b} (H^*, p)$ which combines blocks $a$ and $b$ in $H^*$ by taking weighted averages with proportions in $p$. For example, for $H=U_{K-1,K}(H^*, p)$, 
\begin{align}
H_{l,k}  & = H^*_{l,k} \quad \text{for } 1 \leq l,k \leq K-2;	\notag \\ 
H_{l,K-1} & = \frac{p_l p_{K-1} H^*_{l,K-1} + p_l p_K H^*_{l,K}}{p_l p_{K-1}+p_l p_{K}} \quad \text{for }	1\leq l \leq K-2;	\notag \\
H_{K-1,K-1} & = \frac{p_{K-1}^2 H^*_{K-1,K-1} + 2p_{K-1} p_K H^*_{K-1,K} + p_K^2 H^*_{K,K}}{p_{K-1}^2 + 2p_{K-1} p_K + p_K^2}.
\end{align}
A schematic representation of $H$ is given in Figure \ref{fig:merge}.

\begin{figure}[h!]
\begin{centering}
\includegraphics[width=10cm]{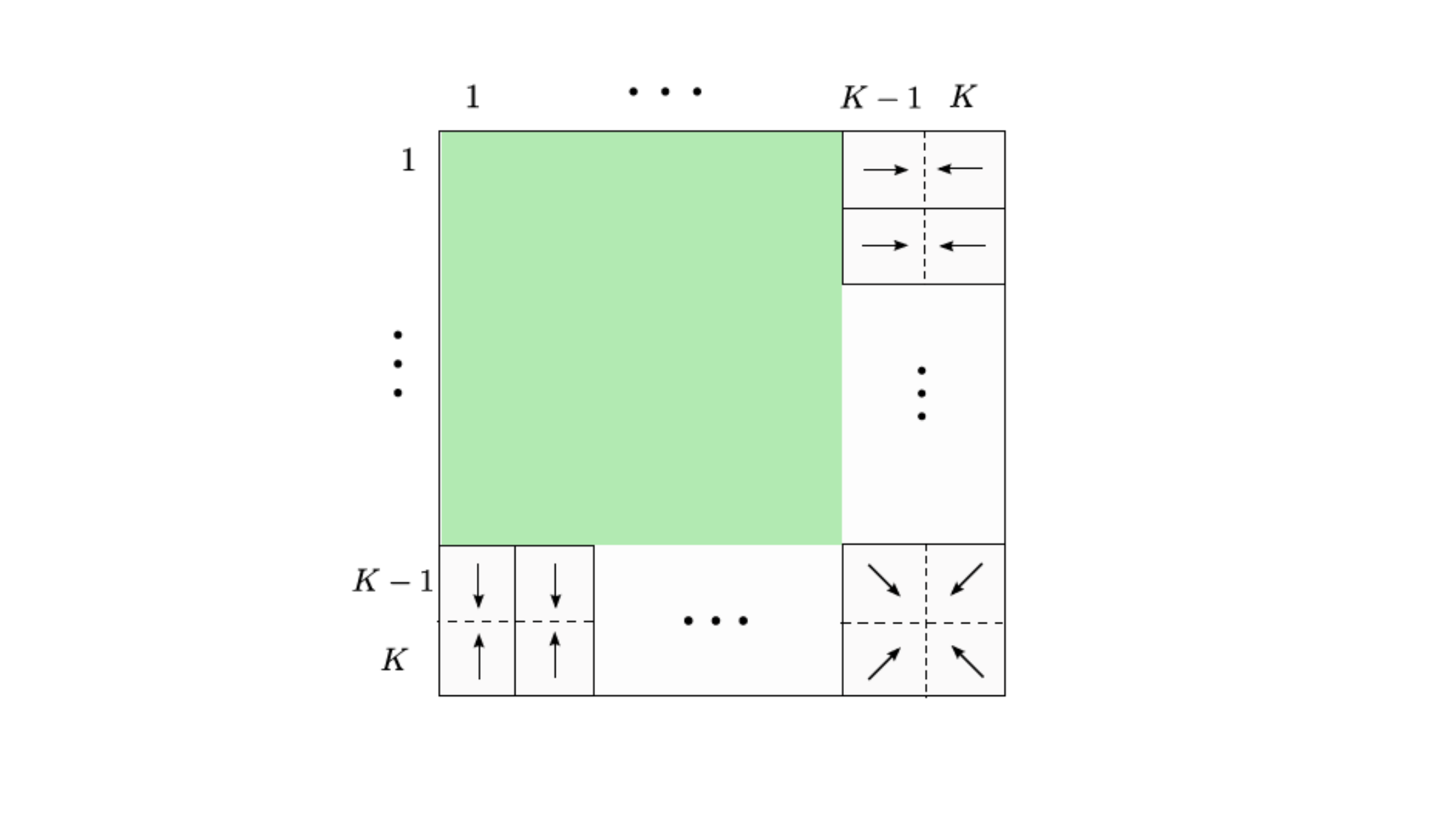}
\caption{A schematic representation of how $H^*$ is merged to give $H=U_{K-1,K}(H^*, p)$. The green area contains unchanged parameters and the arrows indicate where mergings occur.}
\label{fig:merge}
\end{centering}
\end{figure}

For consistency, when merging two blocks $(a,b)$ with $b>a$, the new merged block will be relabeled $a$ and all the blocks $c$ with $c>b$ will be relabeled $c-1$. Using this scheme, we also obtain the merged node labels $U_{a,b}(Z)$ and merged proportions $U_{a,b}(p)$ with $[U_{a,b}(p)]_a = p_a + p_b$. 

Constraining the parameters to a smaller model results in a suboptimal likelihood and its distance from the likelihood associated with the correct model can be measured by the Kullback-Leibler divergence, denoted $D_{KL}(\cdot \Vert \cdot)$. Let
\begin{align*}
\gamma_1(x) & = x\log x + (1-x)\log(1-x),	\\
\gamma_2(x) & = x\log x - x.
\end{align*}
and define \begin{equation}
D_i (a,b) = \sum_{k,l=1}^{K-1}  [U_{a,b}(\pi^*)]_k [U_{a,b}(\pi^*)]_l \gamma_i ([U_{a,b}(H^*, \pi^*)]_{k,l}) 
\label{eq_D}
\end{equation}
When $p=\pi^*$ and treating the labels $Z$ as fixed parameters, denote $P_{A \mid Z, H^*}$ the probability distribution of $A$. Then the information loss incurred by the merging operation $U_{a,b}$ can be measured by 
\begin{align}
 & D_{\text{KL}} \left( P_{A\mid Z, H^*} \Vert P_{A \mid U_{a,b}(Z), U_{a,b}(H^*, \pi^*)}\right)	\\
 = &\begin{cases} \frac{n^2}{2} \left[ \sum_{c,d=1}^{K} \pi_c^*\pi_d^* \gamma_1(H^*_{c,d})  - D_1(a,b) \right]  + O(n),	 \,\, \text{ for } \rho_n=\Omega(1);		\notag\\
  \frac{n^2\rho_n}{2} \left[ \sum_{c,d=1}^{K} \pi_c^*\pi_d^* \gamma_2(H^*_{c,d}) - D_2(a,b) \right] + O(n^2\rho_n^2),  	\,\, \text{ for } \rho_n \to 0. 
 \end{cases}
 \label{eq_kl}
 \end{align}
 Thus an optimal merging minimizing $D_{\text{KL}}$ is essentially equivalent to maximizing $D_i(a,b)$. 
 
We assume the following holds for $\theta^*$:
\begin{assumption}
\label{assump_unique}
A unique maximum exists for $\max_{(a,b)} D_i(a,b)$.
\end{assumption}
This assumption is more of a notational convenience than necessity. From now on without loss of generality assume the maximum is achieved at $a=K-1$ and $b=K$, and denote $H' = U_{K-1, K}(H^*, \pi^*)$, $S' = H' / \rho_n$ and $Z'=U_{K-1, K} (Z)$. We also assume $H'$ is identifiable in the sense that
\begin{assumption}
$S'$ has no identical columns.
\label{assump_iden}
\end{assumption}
Thus the merged model cannot be collapsed further to a smaller model. 

The next lemma argues $\sup_{\theta\in\Theta_{K-1}} g(A; \theta)$ is essentially dominated by the complete likelihood associated with the optimal merging. 

\begin{lemma}
\label{lem_bounds}
Let $\S(z)$ be the set of labels which are equivalent up to a permutation $\tau$, $\S(z) = \{\tau(z) \mid \tau: [K-1] \to [K-1]\}$. 
Then
\begin{equation}
\sum_{z\notin\S(Z')} \sup_{\theta\in\Theta_{K-1}} f(z, A; \theta) =  \sup_{\theta\in\Theta_{K-1}} f(Z', A; \theta) o_P(1).
\label{eq_label_bound}
\end{equation}
\end{lemma}
The proof is shown in the Appendix. 

This lemma provides a tractable bound on $\sup_{\theta\in\Theta_{K-1}} g(A; \theta)$, allowing the rest of the analysis to be carried out by usual Taylor expansion. 

Define 
\begin{align}
\mu_1(\theta^*) & =  \frac{1}{2} \left[ D_1 (K-1, K) -\sum_{c,d=1}^{K} \pi_c^*\pi_d^* \gamma_1(H^*_{c,d}) 	\right]	\notag\\
 \mu_2 (\theta^*) & = \mu_1(\theta^*) + \frac{1}{n} \left\{ (\pi^*_{K-1} + \pi^*_{K}) \log(\pi^*_{K-1} + \pi^*_{K}) - \pi^*_{K-1} \log  \notag\pi^*_{K-1} -  \pi^*_{K} \log \pi^*_{K} \right\}
\end{align}
Upon merging blocks $K-1$ and $K$, denote $u(a)$ as the new block label of block $a$, and define $d_i(a,b)$ such that 
\begin{align}
d_1(a,b) & =H^*_{a,b}\log\frac{H'_{u(a),u(b)}}{H^*_{a,b}} + (1-H^*_{a,b})\log\frac{1-H'_{u(a),u(b)}}{1-H^*_{a,b}}		\notag\\
d_2(a,b) & = S^*_{a,b}\log\frac{S'_{u(a),u(b)}}{S^*_{a,b}} + (S'_{u(a),u(b)} - S^*_{a,b}).
\label{eq_def_d}
\end{align}
The following theorem gives the asymptotic distribution of $L_{K, K-1}$, the proof of which is shown in the Appendix.
\begin{theorem}
Suppose the underlying model parameter generating $A$ is $\theta^*=(\pi^*, H^*) \in \Theta_K$, then $L_{K,K-1}$ is asymptotically normal with
\begin{align}
n^{-3/2} L_{K,K-1} - \sqrt{n}\mu_1(\theta^*) & \stackrel{D}{\longrightarrow} N(0,\sigma_1^2(\theta^*)), \quad \text{ if } \rho_n = \Omega(1);		\notag\\
\rho_n^{-1}n^{-3/2} L_{K,K-1} - \rho_n^{-1}\sqrt{n}\mu_2 (\theta^*) & \stackrel{D}{\longrightarrow} N(0,\sigma_2^2(\theta^*)), \quad \text{ if } \rho_n \to 0.
\label{eq_asy_gaussian}
\end{align}
when Assumptions \ref{assump_unique} and \ref{assump_iden} hold. Let $\Sigma(\pi^*)$ be the covariance matrix of a $multinomial(\pi^*)$ distribution, that is $\Sigma_{a,a}(\pi^*) = \pi_a^*(1-\pi_a^*)$ and $\Sigma_{a,b} = -\pi_a^*\pi_b^*$ for $a\neq b$. The variance $\sigma_i^2(\theta^*)$ is given by $J^T(\theta^*) \Sigma(\pi^*) J(\theta^*)$ for $i=1, 2$, where $J(\theta^*) = (J_b(\theta^*))_{1\leq b\leq K}$, \[
J_b(\theta^*) = 
\begin{cases}
\pi_{K-1}^*d_i(b, K-1) + \pi_K^*d_i(b, K)	\quad \text{for }	1\leq b\leq K-2	\\
\sum_{a\neq b} \pi_a^*d_i(a,b)  + \pi_b^*d_i(b,b)  \quad \text{for }b=K-1,K.	
\end{cases}
\] 
\label{thm_underfit}
\end{theorem}

\begin{remark}
(i) A special case occurs for $K=2$, $\pi_1^* = \pi_2^*$, $H^*_{1,1} = H^*_{2,2}$. In this case, $\sigma^2_i(\theta^*) = 0$ and $\rho_n^{-1}n^{-3/2}L_{K,K-1}$ converges to its asymptotic mean. In general, for homogeneous block models with $H^*_{a,a} = h$ and $H^*_{a,b} = g$ for $a\neq b$,  $J(\theta^*)$ simplifies to $J_b(\theta^*) = 0$ for $b\leq K-2$, $J_{K-1}(\theta^*) = \pi^*_{K-1}d_i(K,K) + \pi^*_{K}d_i(K-1,K)$, and $J_{K}(\theta^*) = \pi^*_{K-1}d_i(K-1,K) + \pi^*_{K}d_i(K,K)$.

(ii) In general, underfitting a $K^- < K$ model will lead to the same type of limiting distribution under conditions similar to Assumptions \ref{assump_unique} and \ref{assump_iden}, assuming the uniqueness of the optimal merging scheme and identifiability after merging. That is,
\begin{equation}
\rho_n^{-1}n^{-3/2} L_{K,K^-} - \rho_n^{-1}\sqrt{n}\mu \stackrel{D}{\longrightarrow} N(0, \sigma^2)
\end{equation}
for some mean $\mu \sim C\rho_n$ and variance $\sigma^2$. The proof will be similar but involve more tedious descriptions of how various merges can occur.

(iii) The asymptotic distributions derived under the null distribution of a $K$-block model suggest one might consider performing hypothesis testing directly to compare against an alternative simpler model. However, the asymptotic means depend on the true parameters, and its maximum likelihood estimate converges only at the rate $\sqrt{n}$ \citep{bickel:2013}.

(iv) Without Assumptions \ref{assump_unique} and \ref{assump_iden}, it is easy to show 
\begin{equation}
L_{K,K^-} \leq -\Omega_P(n^2\rho_n),
\label{eq_general_underfit}
\end{equation}
where $\Omega(\cdot)$ denotes asymptotic lower bound, using the method in proving Theorem \ref{thm_overfit}. 
\label{rem_underfit}
\end{remark}

\subsection{Overfitting} 
\label{sec_over}
In the case of overfitting a $K^+$-block model with $K^+ > K$, deriving the asymptotic distribution of $L_{K, K^+}$ is much more challenging. Intuitively, embedding a $K$-block model in a larger model can be achieved by appropriately splitting the labels $Z$ and there are an exponential number of possible splits. We first show a result analogous to Lemma \ref{lem_bounds}. However, the number of summands involved in $\sup_{\theta\in\Theta_{K^+}} g(A;\theta)$ remains exponential this time. 

Recall that for $z\in[K^+]^n$, $R(z, Z)$ is the $K^+\times K$ confusion matrix. We first define a subset $\sV_{K^+}\in [K^+]^n$ such that 
\[
\sV_{K^+} = \left\{ z\in [K^+]^n \mid \text{ there is at most one nonzero entry in every row of $R(z, Z)$} \right\}.
\]
$\sV_{K^+}$ is obtained by splitting of $Z$ such that every block in $z$ is always a subset of an existing block in $Z$. The next lemma shows it suffices to consider only the subclass of labels $\sV_{K^+}$ in the sum $g(A;\theta)$, the proof of which is given in the Appendix. 

\begin{lemma}
Suppose $\theta^*\in\Theta_{K}$, then
\[
\sum_{z\in[K^+]^n} \sup_{\theta\in\Theta_{K^+}} f(z, A;\theta) = (1+ o_P(1))\sum_{z\in\sV_{K^+}} \sup_{\theta\in\Theta_{K^+}} f(z,A;\theta).
\]
\label{lem_overfit_split}
\end{lemma}

The lemma does not provide a direct simplification of the sum and suggests the reason why obtaining an asymptotic distribution for $L_{K, K^+}$ is difficult. On the other hand, with appropriate concentration we can still derive the asymptotic order of the statistic.

\begin{theorem}
Suppose $\theta^* \in \Theta_{K}$, then overfitting by a $K^+$-block model with $K^+ >K$ gives $L_{K, K^+} = O_P(n^{3/2}\rho_n^{1/2})$.
\label{thm_overfit}
\end{theorem}

The proof is provided in the Appendix. 

\subsection{Model selection}
\label{sec_md}
The results in the previous sections lead us to construct a penalized likelihood criterion for selecting the optimal block number. The criterion is consistent in the sense that asymptotically it chooses the correct $K$ with probability one. Define
\begin{equation}
\beta(K') = \sup_{\theta\in\Theta_{K'}} \log g(A; \theta) - N_{K'} B_n,
\end{equation}
where $B_n$ gives the order of the penalty term, and $N_{K'}$ is a strictly increasing sequence indexed by $K'$ describing the complexity of the model. The optimal $K_0$ is such that 
\begin{equation}
K_0 = \arg\max_{K'} \beta(K').
\end{equation}

\begin{corollary}
For $K' < K$,  setting $B_n = o(n^2\rho_n)$, 
\begin{equation}
\P_{\theta^*}(\beta(K') < \beta(K)) \to 1. 
\end{equation}
For $K' > K$, setting $B_n$ such that $ B_n n^{-3/2}\rho_n^{-1/2} \to \infty$, 
\begin{equation}
\P_{\theta^*}(\beta(K') < \beta(K)) \to 1. 
\end{equation}
\label{cor_md}
\end{corollary}

\begin{proof}
For $K' < K$, generalizing Theorem \ref{thm_underfit},
\begin{align}
& \P_{\theta^*} \left( \beta(K') < \beta(K) \right) 	\notag\\
= & \P_{\theta^*} \left( n^{-3/2}\rho_n^{-1} \log \frac{\sup_{\theta\in\Theta_{K'}} g(A; \theta)}{\sup_{\theta\in\Theta_{K}} g(A; \theta)} - \sqrt{n}\rho_n^{-1}\mu \right.	\notag\\
 & \qquad\qquad\qquad \left. < (N_{K'} - N_{K}) \frac{B_n}{n^{3/2}\rho_n} - \sqrt{n}\rho_n^{-1}\mu \right) 	\notag\\
\to & 1,
\label{eq_under_probability}
\end{align}
since $B_n = o(n^{2}\rho_n)$ and $-\rho_n^{-1}\mu \geq C(\theta^*)$ for some positive constant depending on $\theta^*$. In general the same conclusion holds by Remark \ref{rem_underfit} (iii).

For $K' > K$, using Theorem \ref{thm_overfit},
\begin{align}
& \P_{\theta^*} \left( \beta(K') < \beta(K) \right) 	\notag\\
= & \P_{\theta^*} \left( \frac{1}{n^{3/2}\rho_n^{1/2}} \log\frac{\sup_{\theta\in\Theta_{K'}} g(A; \theta)}{\sup_{\theta\in\Theta_{K}} g(A; \theta)} < (N_{K'} - N_{K}) \frac{B_n}{n^{3/2}\rho_n^{1/2}}\right)	\notag\\
\to & 1,
\label{eq_over_probability}
\end{align}
when $ B_n n^{-3/2}\rho_n^{-1/2} \to \infty$. 
\end{proof}

Since the ratio of the upper bound $n^2\rho_n$ and the lower bound $n^{3/2} \rho_n^{1/2} $ tends to infinity, such a sequence $B_n$ exists. Choosing $B_n$ in this interval, we have $K_0=K$ with probability tending to 1. However, we also note that for finite cases with moderate-sized $n$, $\sqrt{n} \rho_n^{-1}\mu$ in \eqref{eq_under_probability} is small, making it easy to over penalize with large $B_n$. At the same time, the lower bound in Theorem \ref{thm_overfit} is not tight and can be refined further. 

We further assume the following holds for tractable approximation.
\begin{assumption}
The maximum is achieved in the set $\sN_{K^+} = \{z\in\sV_{K^+} \mid n_k(z) \geq \epsilon n \text{ for all }k, \text{ for some }\epsilon >0,  \}$. 
\label{assump_balance_design}
\end{assumption}

Assumption \ref{assump_balance_design} assumes the maximum can only be achieved on a loosely balanced block design. The assumption and Lemma \ref{lem_overfit_split} imply it remains to analyze the order of $\max_{z\in\sN_{K^+}} \sup_{\theta\in\Theta_{K^+}} \log f(z,A;\theta)$. The following theorem shows the order of $L_{K, K^+}$ can be refined to $O_P(n)$. The details can be found in the Appendix. 

\begin{theorem}
Under Assumption \ref{assump_balance_design}, $L_{K,K^+}$ is of order $O_P(n)$ for $K^+>K$.
\label{thm_overfit_refine}
\end{theorem}

It follows then choosing $B_n$ growing slightly faster than $n$ will ensure consistency in the sense described in Corollary \ref{cor_md}. Thus we choose a penalized likelihood of the following form,
\begin{equation}
\beta(K') = \sup_{\theta\in\Theta_{K'}} \log g(A; \theta) - \lambda \cdot \frac{K'(K'+1)}{2}n\log n,
\label{eq_criterion}
\end{equation}
where the complexity term corresponds to the number of parameters in the edge probability matrix and the constant $\lambda$ is a tuning parameter. Similar to many BIC-type criteria, choosing the tuning parameter is a challenging problem even though it does not affect the asymptotic properties. We discuss this problem in Section \ref{sec_sim}.

\subsection{Extension to a degree-corrected stochastic block model}
\label{subsec_dcsbm}
In practice, the SBM often oversimplifies the community structures by assuming all the nodes within a block have the same expected degree, thereby excluding networks with ``hub'' nodes and other possible degree variations within blocks. To address this limitation, \citet{KarrerNewman:2010} proposed the degree-corrected stochastic block model (DCSBM) by setting 
\begin{equation}
\E\left( A_{i,j} \mid Z, \omega \right) = \omega_i \omega_jH_{Z_i, Z_j}, \quad i\neq j
\end{equation}
where $\omega=(\omega_1,\dots, \omega_n)$ is the set of node degree parameters with some identifiability constraint. 

As before, $Z\sim\text{Multinomial}(\pi)$. We also treat $\omega$ as a latent variable and assume $\omega_i\mid Z \sim n_k(Z) \cdot \text{Dirichlet}(\1)$ for $Z_i = k$ so that $\omega$ satisfies the identifiability constraint $\sum_{i:Z_i = k} \omega_i = n_k$ for every $k$. Similar to \citet{KarrerNewman:2010}, we replace the Bernoulli likelihood by the Poisson likelihood and assume $A_{i,i} \sim\text{Poisson}(\omega_i^2 H_{Z_i, Z_i}/2)$ to simplify derivation. As noted in \citet{KarrerNewman:2010} and \citet{zhao:2012}, sparse networks are well approximated by the Poisson distribution and little difference was found in practice between the two choices. The assumption on the diagonal entries also does not change the asymptotic results. Therefore given $(Z, \omega)$, the log conditional likelihood of $A$ is (up to a constant)
\begin{align} 
\log f(A \mid Z, \omega; \theta) & = \frac{1}{2}\sum_{i,j} \left( A_{i,j} \log (\omega_i \omega_j H_{Z_i, Z_j}) - \omega_i \omega_j H_{Z_i, Z_j} \right) 	\notag\\
 & =  \sum_{i,j} A_{i,j} \log \omega_i + \frac{1}{2}\sum_{k,l} \left( O_{k,l}(Z) \log H_{k,l}  - n_k(Z) n_l(Z) H_{k,l}\right).
\end{align}

In this case, the likelihood function $f(Z, A; \theta)$ has a tractable form and one can show Lemma \ref{lem_bounds} holds provided $n^{1/2}\rho_n/\log n \to\infty$. The stricter condition on the degree density ensures even with node degree variations (in the worst case $\E(A_{i,j} \mid Z, \omega) \sim\rho_n/n$) there still exist enough edges for parameter estimation. Similar arguments apply to show that the criterion \eqref{eq_criterion} is asymptotically consistent for this DCSBM. As the derivation is largely similar to the regular SBM case, we provide a proof sketch in the supplementary material. 


\subsection{Likelihood approximations}
In practice, direct computations of the likelihood function $g(A; \theta)$ and its supremum involve an exponential number of summands and quickly become intractable as $n$ grows. In this section we provide practical ways to approximate the likelihood and discuss conditions under which asymptotic consistency is preserved. 

\subsubsection*{Variational likelihood for regular SBM}
Using the EM algorithm to optimize over $\theta$ requires computing the conditional distribution of $Z$ given $A$, which is not factorizable in this case. Variational methods tackle the true conditional distribution $f_{Z \mid A; \theta}$ with the mean field approximation, thus simplifying the local optimization at each iteration. Under the regular SBM, the variational log likelihood $J(q, \theta; A)$ for a $K'$-block model is defined as
\begin{equation}
J(q,\theta; A) = -D_{KL}(q\Vert f_{Z \mid A; \theta}) + \log g(A;\theta),
\end{equation}
where $q\in \mathcal{D}_{K'}$ is any product distribution with $q(z) = \prod_{i=1}^{n} q_i(z_i)$, $1\leq z_i \leq K'$. The variational estimates $\hat{\theta}_{K'}^{\text{VAR}}$ is given by
\[
\hat{\theta}_{K'}^{\text{VAR}} = \arg\max_{\theta\in\Theta_{K'}} \max_{q\in\mathcal{D}_{K'}} J(q,\theta;A),
\]
which can be optimized using the EM algorithm in \citet{Daudin:2008}. Also we note that $J(q, \theta; A)$ simplifies to
\begin{align*}
J(q, \theta; A) = & \sum_{i=1}^{n} \sum_{k=1}^{K'} q_i(k) (-\log q_i(k) + \log \pi(k)) 	\\
 & + \sum_{i < j} \sum_{k, l =1}^{K'} q_i(k) q_j(l) \left( A_{ij} \log H_{k,l} + (1-A_{ij}) \log (1-H_{k,l}) \right)
\end{align*}
and hence can be easily evaluated. 

We can replace the likelihood in \eqref{eq_criterion} by the variational log likelihood $J$ without changing its asymptotic performance. More precisely, the criterion with variational approximation
\begin{equation}
\beta^{\text{VAR}}(K') = \sup_{\theta\in\Theta_{K'}} \sup_{q\in\mathcal{D}_{K'}} J(q,\theta;A) - \lambda \cdot \frac{K'(K'+1)}{2} n\log n
\label{eq_var_criterion}
\end{equation}
is still asymptotically consistent. Noting that 
\begin{enumerate}
\item[(i)]
$J(q,\theta;A) \leq  \log g(A; \theta)$ for any $q\in \mathcal{D}_{K'}$;

\item[(ii)]
$\sup_{\theta\in\Theta_{K}} \sup_{q\in\mathcal{D}_{K}} J(q,\theta;A) - \sup_{\theta\in\Theta_{K}} \log g(A; \theta) = O_P(1)$ as shown in \cite{bickel:2013},
\end{enumerate}
it can be easily verified that \eqref{eq_under_probability} and \eqref{eq_over_probability} still hold. Although (ii) applies to the global optimum of $J(q,\theta; A)$ which may not be achieved by the EM algorithm, we note that it can be relaxed to accommodate for the difference in practice. Provided the difference between the local optimum found by the algorithm and the global optimum is bounded by $o_P(n\log n)$, asymptotic consistency still holds. 

\subsubsection*{Label estimation}

The computation time of variational likelihood grows quickly with network size and becomes more complicated for degree corrected models. On the other hand, a number of algorithms are available for estimating the latent labels in a computationally efficient way under both regular SBM and DCSBM. Typically these algorithms require specifying the block number, hence let $\hat{Z}(K')$ be the estimated labels corresponding to block number $K'$. Then $f_{ML}(\hat{Z}(K'), A) = \sup_{\theta\in\Theta_{K'}} f(\hat{Z}(K'), A; \theta)$ is the maximum complete likelihood by plugging in the estimated labels. We assume that the estimation algorithm is strongly consistent with the same convergence rate as in Theorem 1 of \citet{BickelChen:2009}.

\begin{assumption}
There exists a sequence $b_n\to\infty$ such that 
\begin{equation}
\P (\hat{Z}(K) \neq \tau (Z) ) = O(n^{-b_n}),
\end{equation}
where $\tau$ is a permutation on $[K]$. 
\label{assump_label_rate}
\end{assumption} 

Such a convergence rate can be achieved by e.g., profile maximum likelihood \citep{BickelChen:2009}. For computational efficiency we will use the pseudo-likelihood algorithm developed by \citet{Amini:2013}, which is available for both regular SBM and DCSBM. Weak consistency for label estimation was shown in \citet{Amini:2013}. However, the plug-in estimates of the block model parameters are still consistent. Under Assumption \ref{assump_label_rate}, observing that 

\begin{enumerate}
\item
$f_{ML}(\hat{Z}(K'), A) \leq \sup_{\theta\in\Theta_{K'}} \log g(A; \theta)$;

\item
$\log f_{ML}(\hat{Z}(K), A) = \sup_{\theta\in\Theta_K}\log g(A; \theta)(1+o_P(1))$,  
\end{enumerate}
it is easy to see \eqref{eq_under_probability} and \eqref{eq_over_probability} still hold, and the criterion
\begin{equation}
\beta^{ML}(K') = \log f_{ML}(\hat{Z}(K'), A) - \lambda\cdot \frac{K'(K'+1)}{2} n \log n
\label{eq_label_criterion}
\end{equation}
is asymptotically consistent. 

In the next section, we use simulated data to demonstrate how the criterion performs in practice. We approximate the likelihood using the variational EM algorithm for regular SBM and the pseudo-likelihood algorithm for DCSBM.


\section{Simulations}
\label{sec_sim}
\subsection{Goodness of fit}
\label{subsec_gof}
We first examined how well the normal limiting distribution approximated the empirical distribution of the statistic in the case of underfitting. Figure \ref{fig_density_cvg} plots the distribution of $n^{-3/2}L_{K,K-1}$ for $n=200$ and $n=500$ obtained from 200 replications for the following two scenarios:
\begin{enumerate}
\item[(a)] $K=2$, $\pi^*=(0.4, 0.6)$, $H^*=\left( \begin{array}{cc}
									0.15 & 0.05	\\
									 	& 0.01	\\
									\end{array}\right)$;
\item[(b)] $K=3$, $\pi^*=(0.4, 0.3, 0.3)$, $H^*=\left( \begin{array}{ccc}
									0.2 & 0.1 & 0.1	\\
										& 0.2 & 0.03	\\
										& 	& 0.1	\\
									\end{array}\right)$.
\end{enumerate}
The log likelihoods are approximated by the variational EM algorithm initialized by regularized spectral clustering \citep{joseph:2013}. The solid curves are normal densities with mean $\mu_2(\theta^*)$ and $\sigma(\theta^*)$ given in Theorem \ref{thm_underfit}. Even though the $O(n)$ term in $\mu_2(\theta^*)$ diminishes asymptotically for $\rho_n$ going to 0 slowly, we found it essential to correct for the bias in the finite sample regimes above. In both cases, the convergence to the Gaussian shape appears faster than the convergence to the mean, and a bias exists for $n=200$. When the network size reaches 500, the empirical distributions are well approximated by their limiting distribution. We note that the bias should not have an adverse effect on model selection since it is in the direction away from zero, making it easier to separate the two models.

\begin{figure}[!h]
\centering
\subfloat[]{\includegraphics[width = 2.5in]{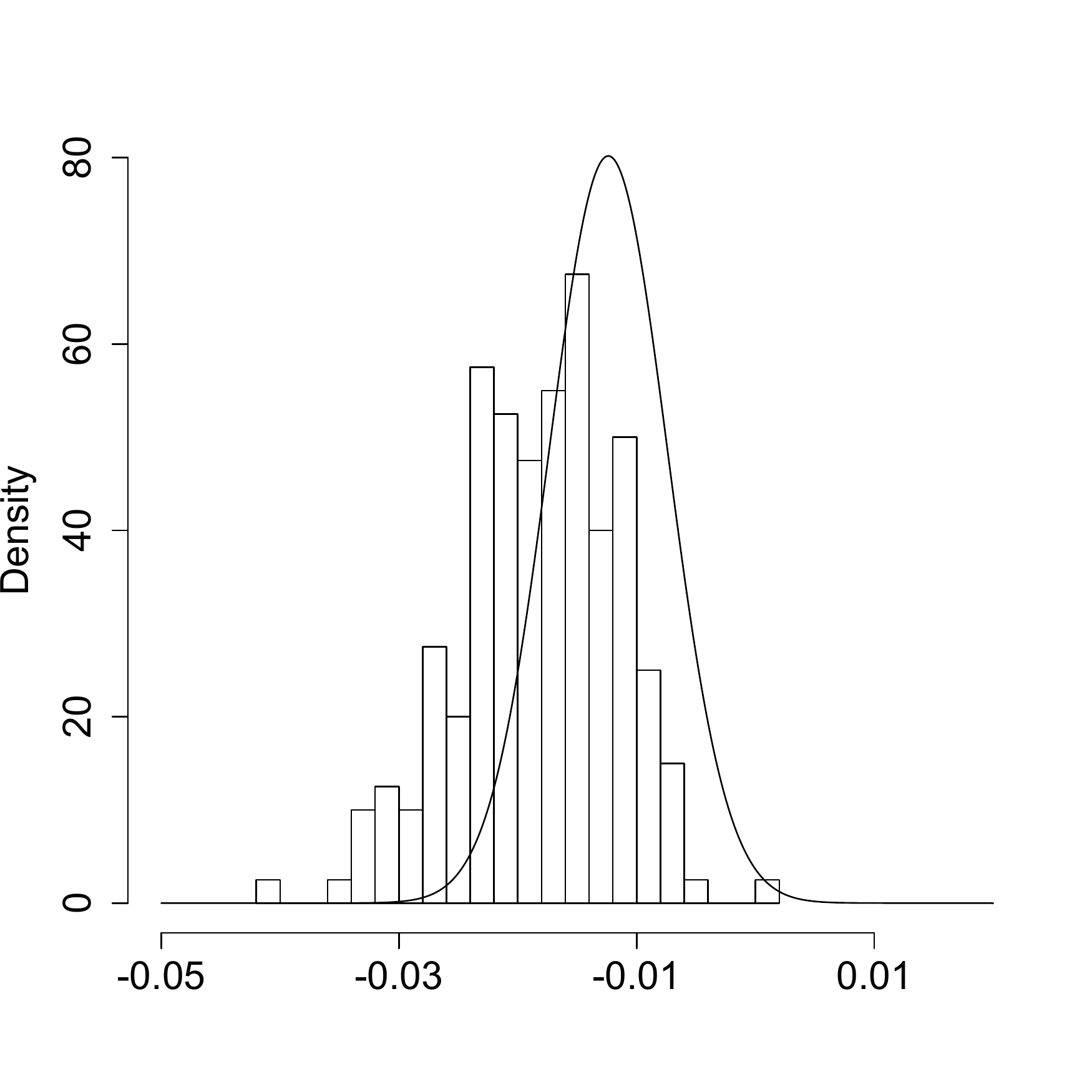}}
\subfloat[]{\includegraphics[width = 2.5in]{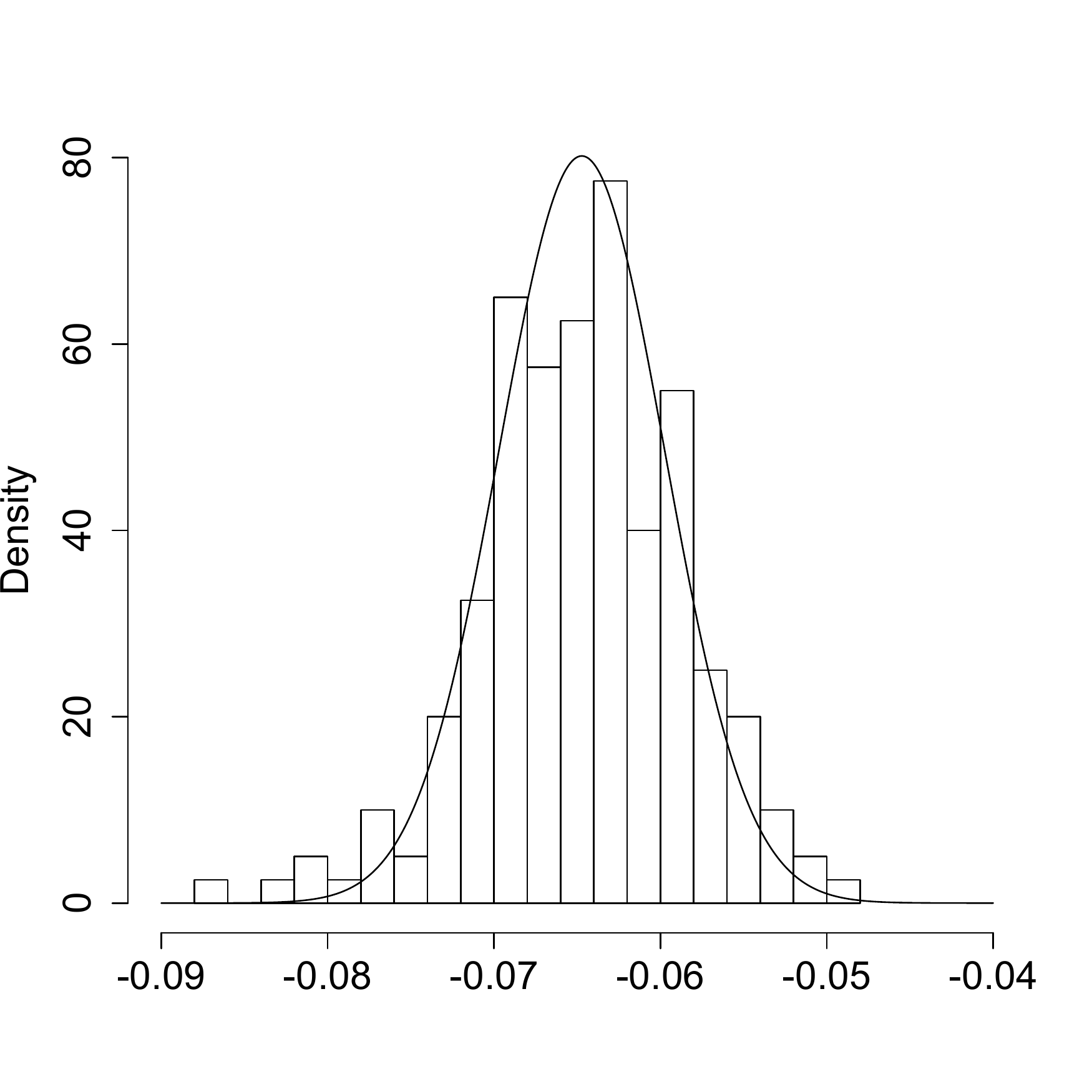}} \\
\subfloat[]{\includegraphics[width = 2.5in]{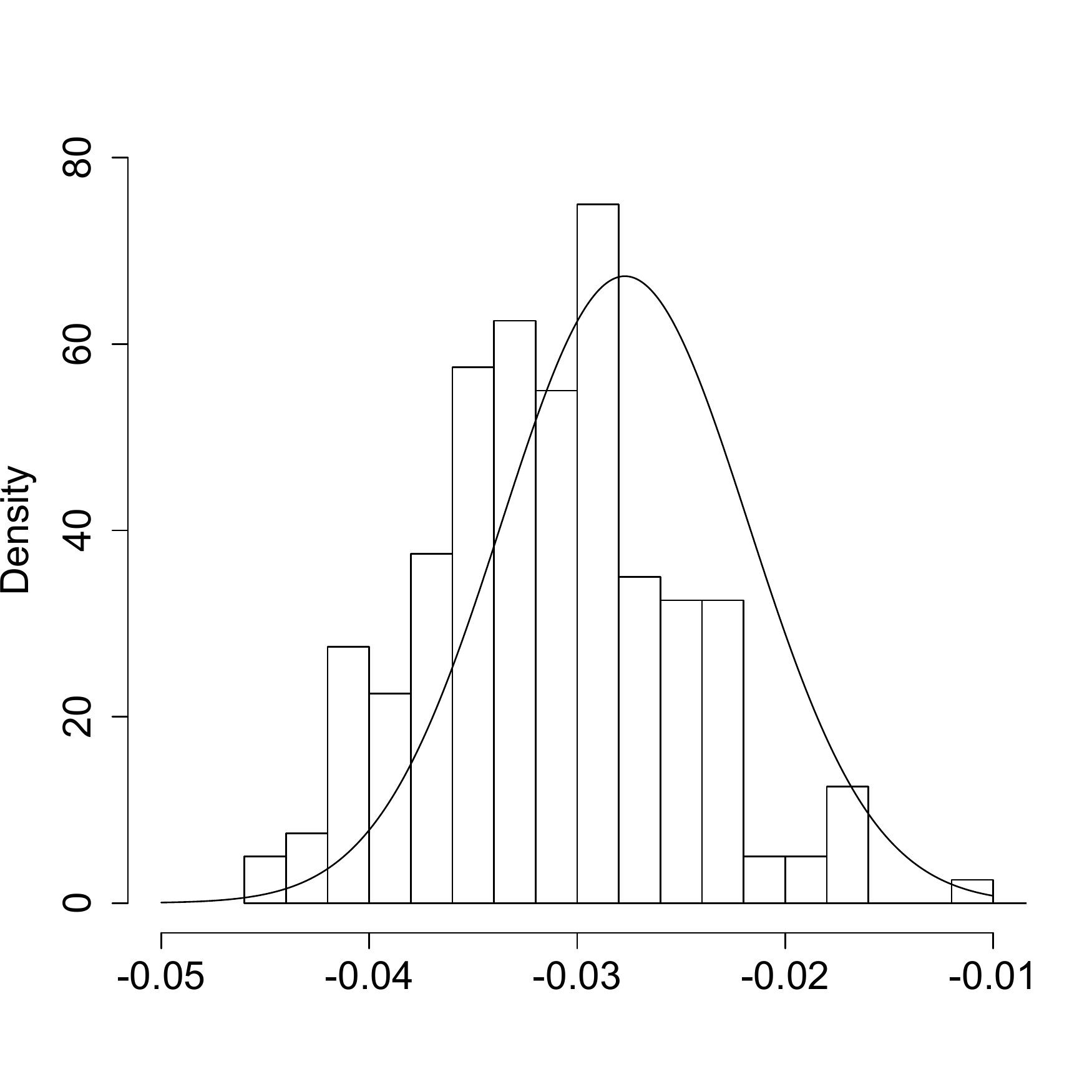}}
\subfloat[]{\includegraphics[width = 2.5in]{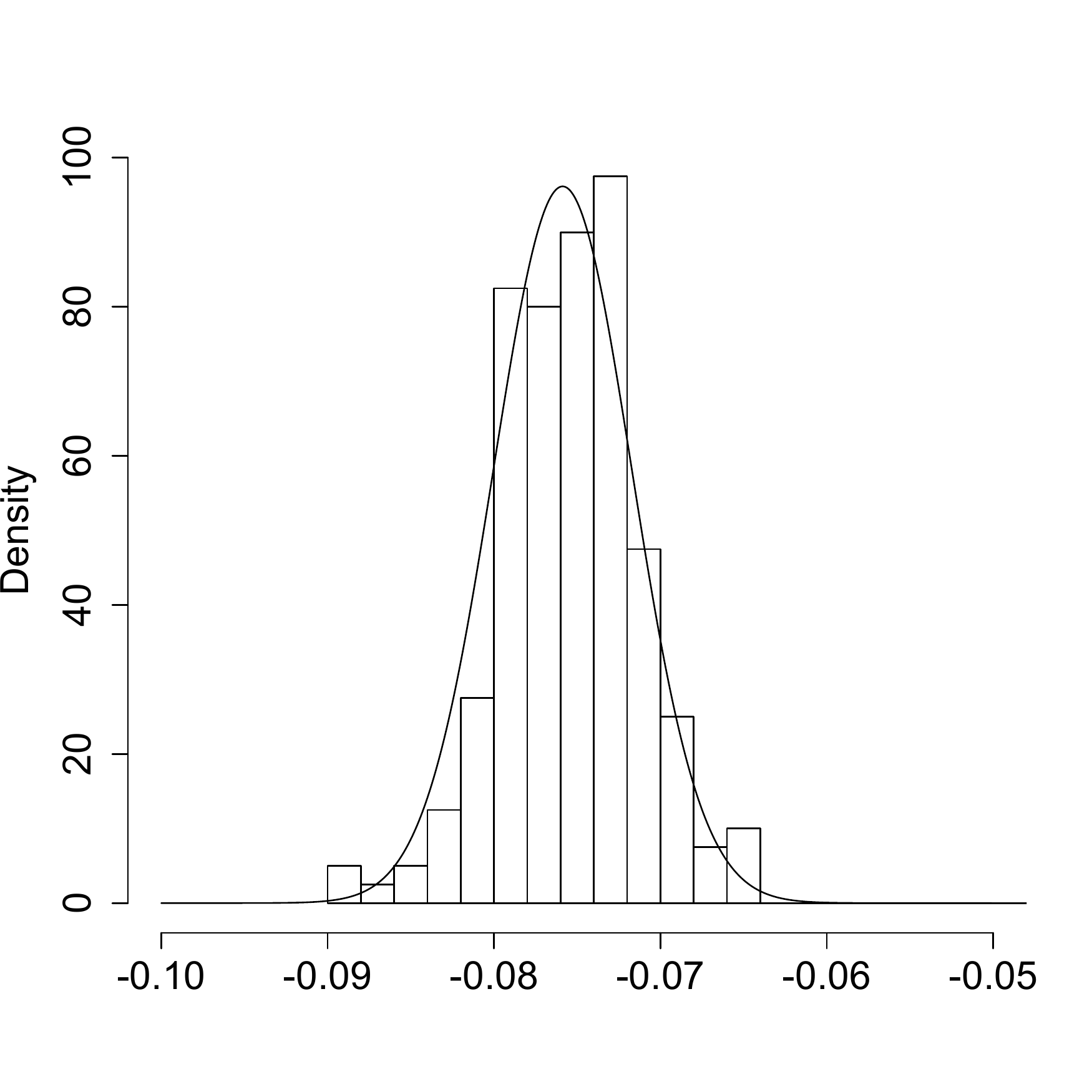}}
\caption{Empirical distributions of $n^{-3/2}L_{K, K-1}$ for (a), (b) K=2, $\pi^*$ and $H^*$ as described in scenario (a); (c), (d) K=3, $\pi^*$ and $H^*$ as described in scenario (b). $n=200$ in (a) and (c); $n=500$ in (b) and (d). The solid curves are normal densities with mean $\mu_2(\theta^*)$ and $\sigma(\theta^*)$ as given in Theorem \ref{thm_underfit}. }
\label{fig_density_cvg}
\end{figure}

\subsection{Selection of tuning parameter}
Before we investigate the finite sample performance of our model selection criterion, we note that \eqref{eq_criterion} involves a tuning parameter $\lambda$. We next propose a heuristic scheme for selecting $\lambda$.  Similar to implementing BIC or AIC-type criteria, a maximum block number $K_{\max}$ to be fitted needs to be chosen first. Then $\lambda$ is chosen by the following algorithm:

\begin{enumerate}
\item
For each choice of $\lambda$, compute $\beta(K')$ for $K' = 1, \dots, K_{\max}$.
\item
Normalize $(-\beta(1), \dots, -\beta(K_{\max}))$ into a probability vector $(w_1, \dots, w_{K_{\max}})$ so that they sum to 1.
\item
Choose $\lambda$ that maximizes the entropy $-\sum_{k} w_k \log(w_k)$.
\item
If ties exist, choose the largest $\lambda$.
\end{enumerate} 
Here $\beta(K')$ can be computed from either the variational likelihood \eqref{eq_var_criterion} or the plug-in maximum likelihood \eqref{eq_label_criterion}. Heuristically, this algorithm chooses a $\lambda$ that maximizes the ``peakedness" of the profile of the penalized likelihoods $(\beta(1), \dots, \beta(K_{\max}))$ and hence the amount of signal contained in it. In the following sections, $\lambda$ was chosen in the interval $[0, 0.3]$ with an increment of $1\times 10^{-3}$; $K_{\max} = 10$ for all simulated data.

\subsection{Performance comparison with other methods}


To see how our criterion (denoted \texttt{plh} for penalized likelihood) performs against other existing model selection methods, we compare its success rate with variational Bayes (\citet{latouche:2012}, denoted \texttt{vb}) and the 3-fold network cross validation method in \citet{chen:2014} (denoted \texttt{ncv}). Since \texttt{vb} is only available for regular SBM, only \texttt{ncv} is included for DCSBM. \texttt{plh} is computed via either the variational EM for regular SBM or the pseudo likelihood algorithm \citep{Amini:2013} for DCSBM.

In Figure \ref{fig_comp} shows the average success rates of all three methods for data generated from regular SBM. 
50 networks of size 500 were generated for each parameter set with $K=2, 3, 4$, $H^* = \rho S^*$, and $\rho\in\{0.02, 0.04, \dots, 0.1\}$. The average degrees of these networks range from around 12 to 75. In general, the success rate of each method decreases as the networks become sparser and $K$ increases, since the task of fitting also becomes harder. Overall \texttt{plh} outperforms the other two methods.

\begin{figure}[!h]
\centering
\subfloat[]{\includegraphics[width = 1.7in]{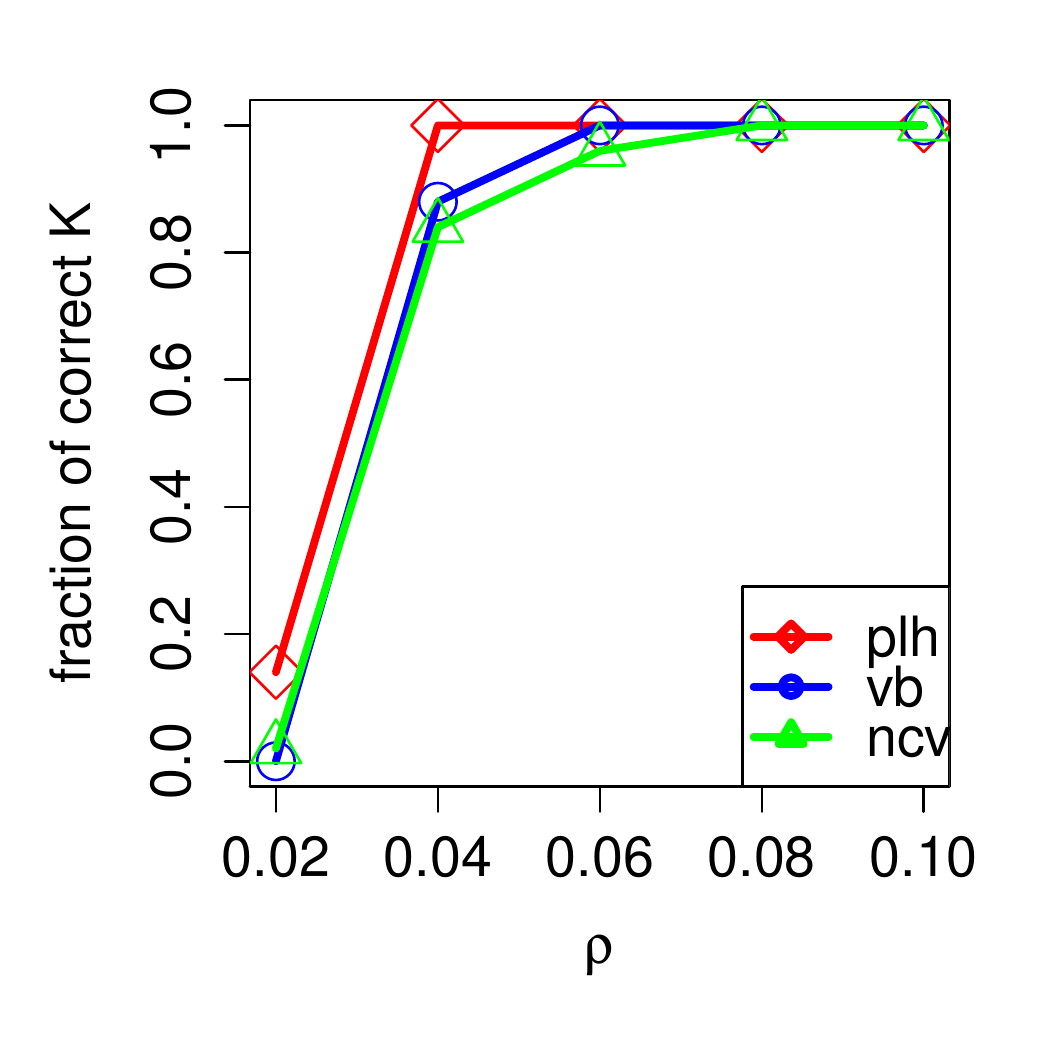}}
\subfloat[]{\includegraphics[width = 1.7in]{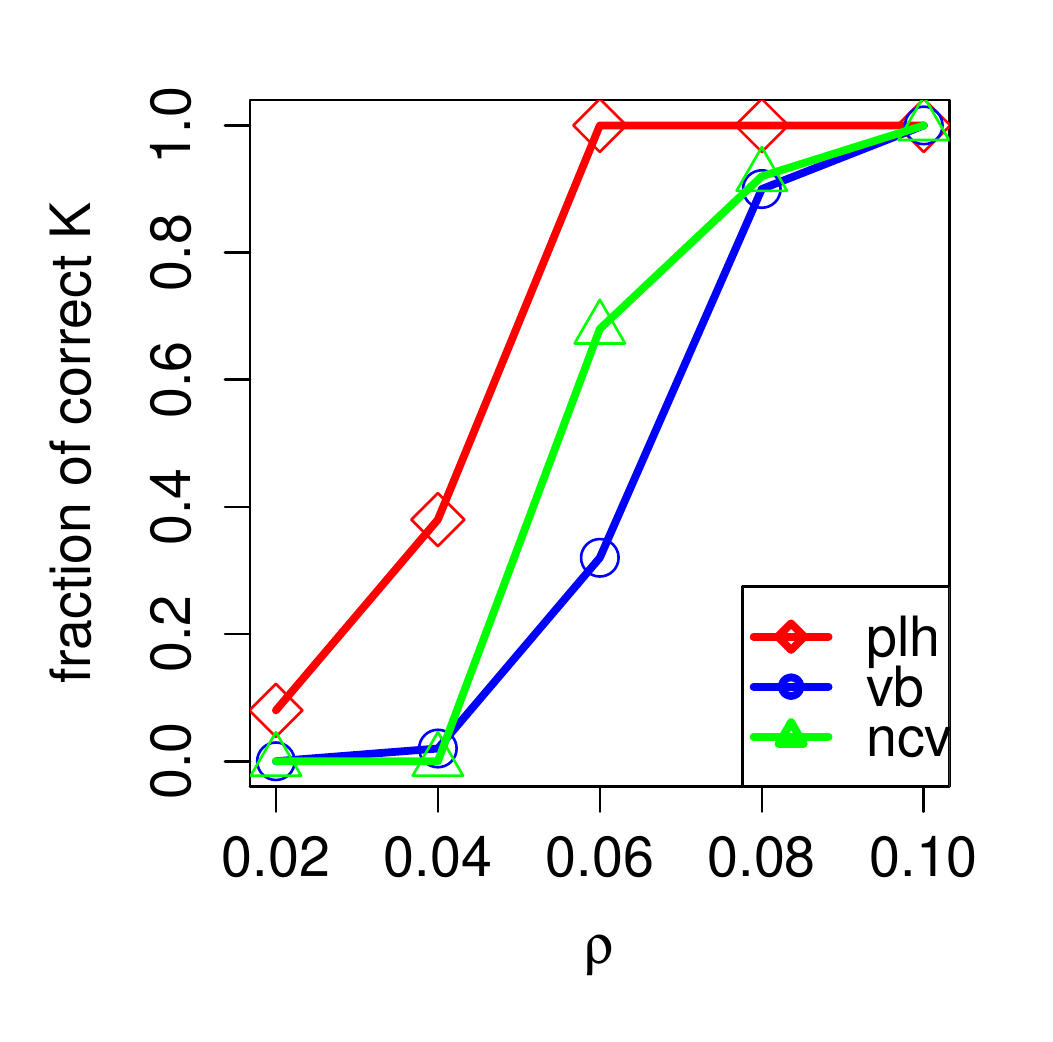}} 
\subfloat[]{\includegraphics[width = 1.7in]{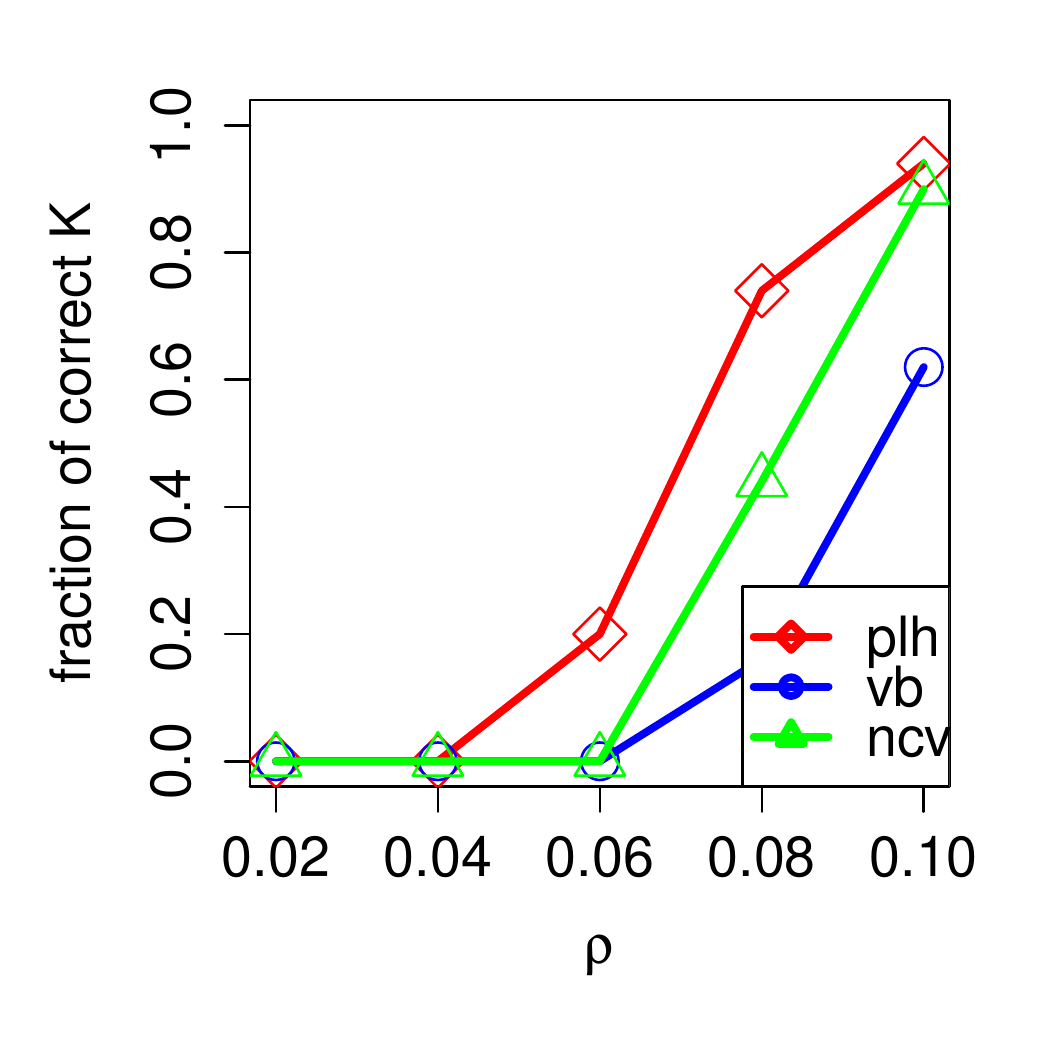}}
\caption{Comparison of the success rates of the penalized likelihood (\texttt{plh}) with variational Bayes (\texttt{vb}) and network cross validation (\texttt{ncv}). For every parameter setting, 50 networks were simulated from regular SBM with (a) $K=2$, $\pi=(0.4, 0.6)$; (b) $K=3$, $\pi=(0.3, 0.3, 0.4)$; (c) $K=4$, $\pi_i=0.25$ for all $i$. In all the cases, $H^* = \rho S^*$, where $\rho\in\{0.02, 0.04, \dots, 0.1\}$, the diagonal elements of $S^*$ equal 2 and the off diagonal elements equal 1.}
\label{fig_comp}
\end{figure}
 

Next we simulated networks from DCSBM. To test if our method also works for more general DCSBM parameter settings and not limited by the specific Dirichlet prior assumption on $\omega$, we generated degree parameters $\omega$ from a Unif(0.2, 1) distribution and further normalized them so that $\sum_{Z_i=k} \omega_i = n_k(Z)$. We set $H^*=\rho S^*$ as in the previous case with varying $\rho$. Binary adjacency matrices were generated even though our model in Section \ref{subsec_dcsbm} is Poisson. As the simulation confirms, the approximation works well when networks are sparse. Table \ref{tab_dcsbm} shows the average success rates of \texttt{plh} (calculated using the pseudo likelihood algorithm) and \texttt{nvc} for 50 networks of size 800 for each parameter set. Overall the problem is harder in this case than regular SBM as the inclusion of degree parameters induces more sparsity in some regions of the networks. \texttt{plh} shows a significant improvement over \texttt{ncv} in almost all cases. 

\begin{table}[h] 
\begin{tabular}{cccccccccc} 
 & \multicolumn{3}{c}{$K=2$} & \multicolumn{3}{c}{$K=3$} & \multicolumn{3}{c}{$K=4$}	\\
$\rho$ & 0.02 & 0.04 & 0.08 & 0.02 & 0.04 & 0.08 & 0.02 & 0.04 & 0.08	\\
\hline
plh	& 0.88 & 0.96 & 1 & 0.08 & 0.66 & 1 & 0.12 & 0.64 & 0.98 	\\
ncv	& 0 & 0.26 & 1 & 0 & 0 & 0.54 & 0 & 0 & 0 	\\
\hline
\end{tabular}
\caption{Comparison of the success rates of the penalized likelihood (\texttt{plh}) with network cross validation (\texttt{ncv}). For every parameter setting, 50 networks were simulated from the DCSBM with (a) $K=2$, $\pi=(0.4, 0.6)$; (b) $K=3$, $\pi=(0.3, 0.3, 0.4)$; (c) $K=4$, $\pi_i=0.25$ for all $i$. In all the cases, $H^* = \rho S^*$, where $\rho\in\{0.02, 0.04, \dots, 0.1\}$, the diagonal elements of $S^*$ equal 2 and the off diagonal elements equal 1. }
\label{tab_dcsbm}
\end{table}

\section{Real world networks}
\label{sec_real}
In this section, we examine the performance of our method on real world networks.  We set $K_{\max}=30$ for the Facebook networks and $K_{\max}=15$ for the others. We first implemented our method along with \texttt{vb} and \texttt{ncv} on nine Facebook ego networks, collected and labeled by \citet{leskovec:2012}. An ego network is created by extracting subgraphs formed on the neighbors of a central (ego) node. Any isolated node was removed before analysis. Fitting regular SBM to these networks, variational EM was used to compute \texttt{plh}. The actual sizes of the networks and the number of communities selected by the three methods are shown in Table \ref{tab_facebook}. The third row of the table shows the number of friend circles in every network with some individuals belonging to multiple circles, but not every individual possesses a circle label. These circle numbers give partial truth on how many communities there are in the networks. Overall \texttt{plh} gives estimates closer to the circle numbers when the network is reasonably large and the number of circles is moderate. \texttt{vb} performs better on networks with a large number of circles but also overfits in a few cases. \texttt{ncv} tends to produce smaller community numbers. 


\begin{table}[h] 
\begin{center}
\begin{tabular}{clllllllll} 
 \# Non-isolated vertices & 333 & 1034  & 224 & 150 & 168 & 61 & 786 & 534 & 52 \\
 \hline
 Average degree & 15 & 52 & 29 & 23 & 20 & 9 & 36 & 18 & 6	\\
 \# Circles & 24 & 9  & 14 & 7 & 13 & 13 & 17 & 32 & 17  \\
\texttt{plh} & 6  & 7 & 6 & 4 & 6 & 6 & 9 & 9 & 6  \\
\texttt{vb} & 11 & 24 & 16 & 9 & 11 & 6 & 25 & 23 & 6	\\
\texttt{ncv} & 3 & 6 & 4 & 2 & 4 & 2 & 2 & 2 & 3  
\end{tabular}
\caption{Facebook ego networks and the number of communities selected by the three methods.}
\label{tab_facebook}
\end{center} 
\end{table}

We also implemented these methods on the political book network \citep{polbooks:2006}, which consists of 105 books and their edges representing co-purchase information from Amazon. Again we treated this as a regular SBM and used variational EM to fit \texttt{plh}. Figure \ref{fig_polbooks} (a) shows the manual labeling of the books based on their political orientations being either ``conservative", ``liberal" or ``neutral". As shown in (b), \texttt{plh} estimated $K=6$ and essentially splits each of the communities in (a) into two, suggesting the presence of sub-communities with more uniform degree distributions.  \texttt{vb} found 4 communities but merged two communities in (a) into one. \texttt{ncv} selected $K=2$ and also merged two clusters in (a).


\begin{figure}[!h]
\centering
\subfloat[]{\includegraphics[width = 1.3in]{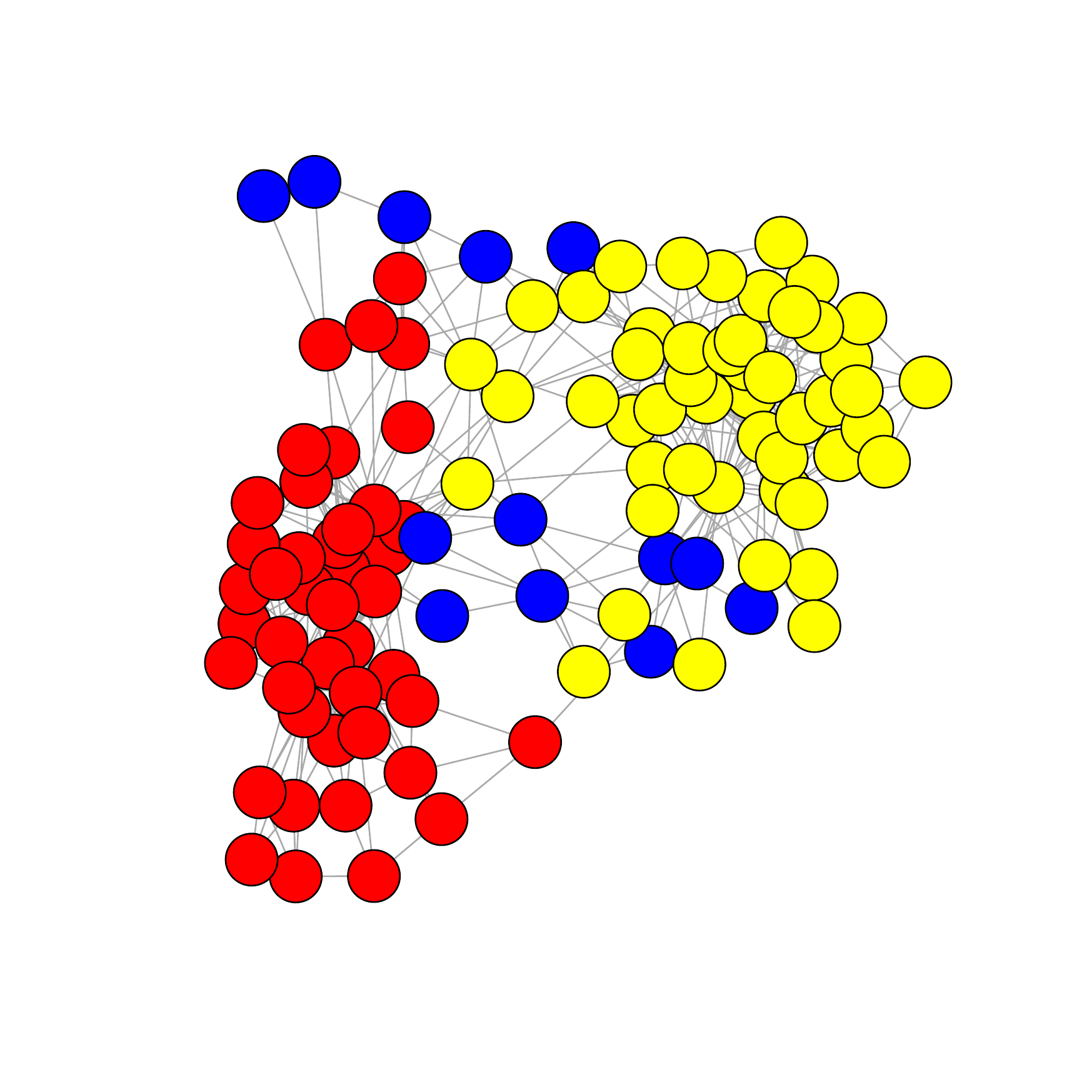}}
\subfloat[]{\includegraphics[width = 1.3in]{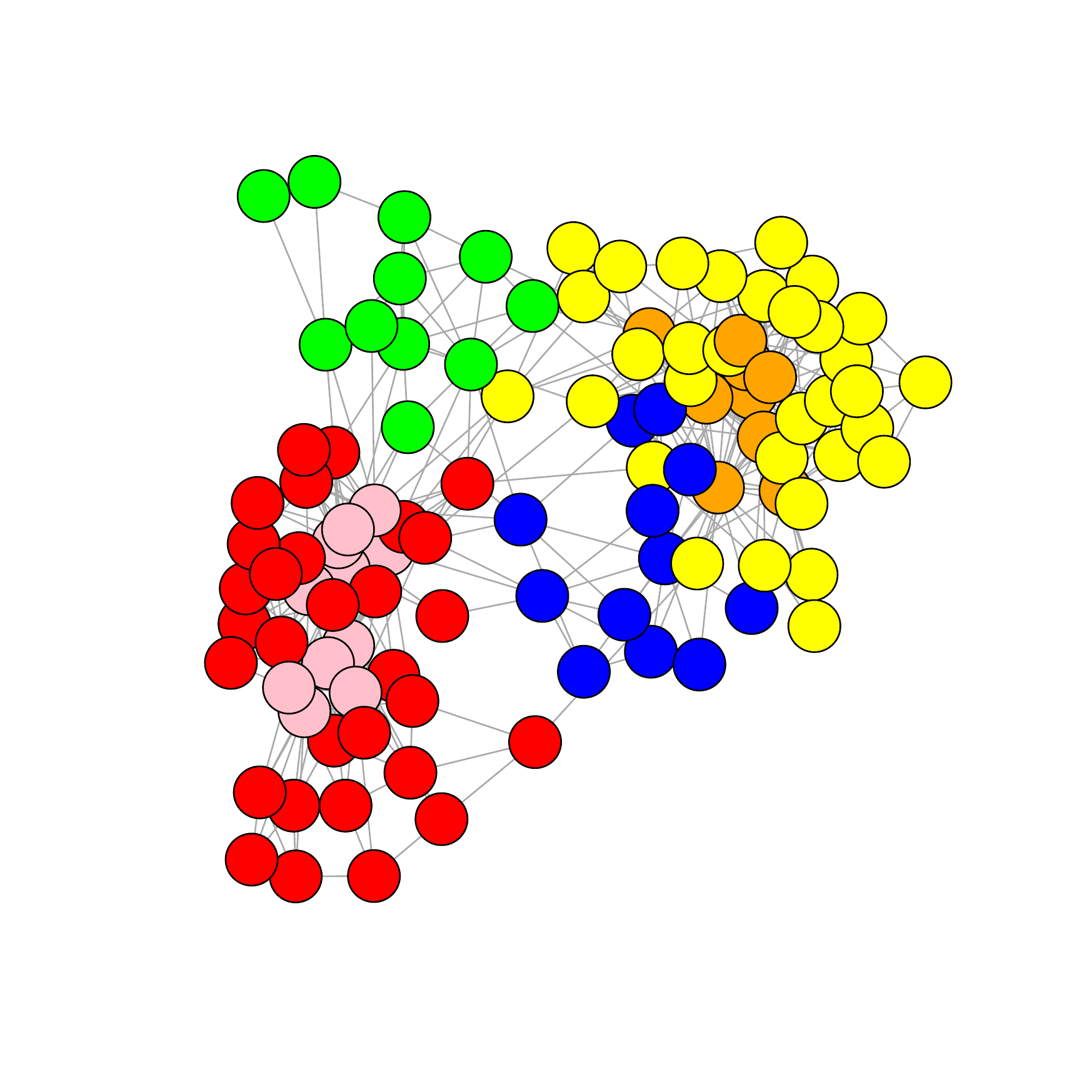}}
\subfloat[]{\includegraphics[width = 1.3in]{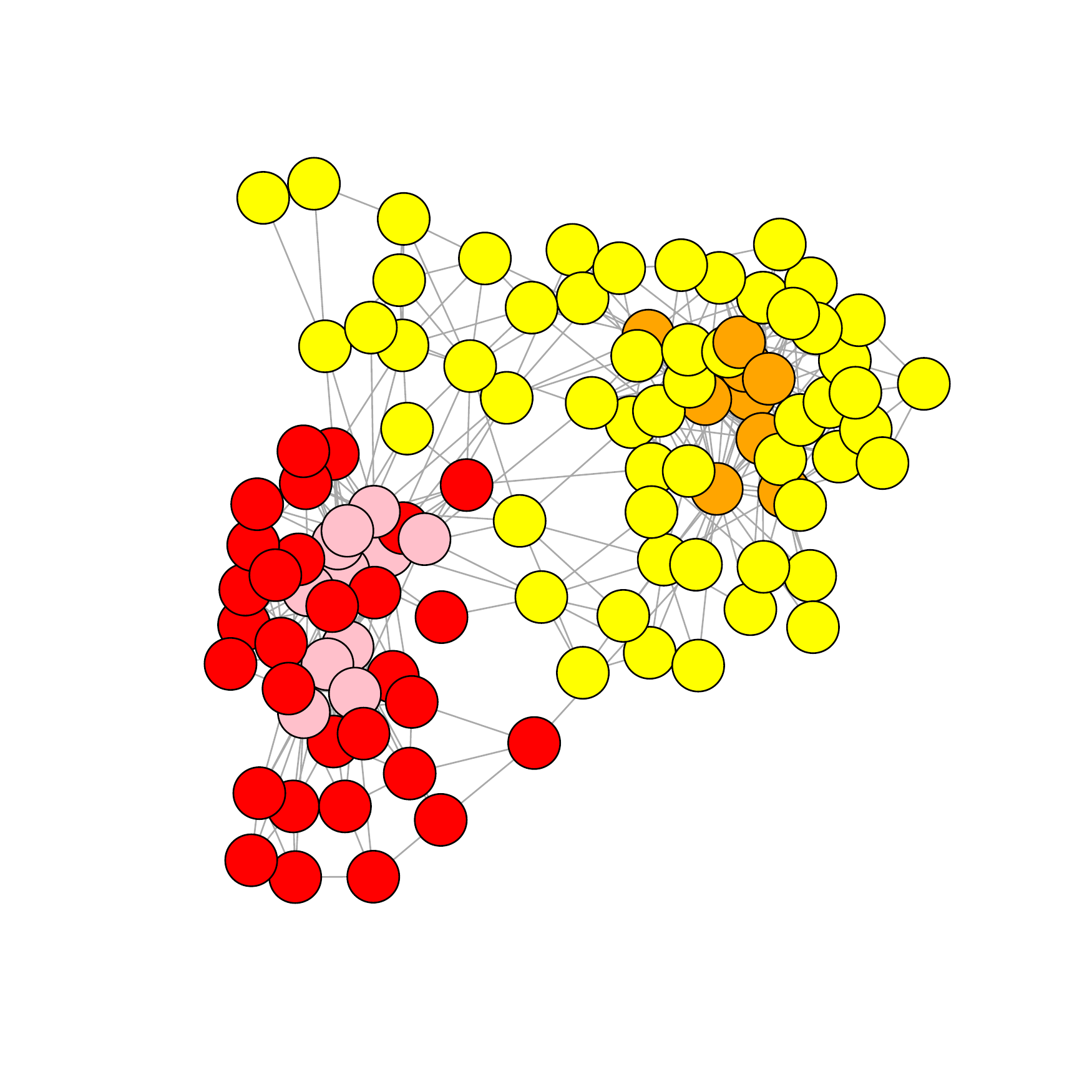}}
\subfloat[]{\includegraphics[width = 1.3in]{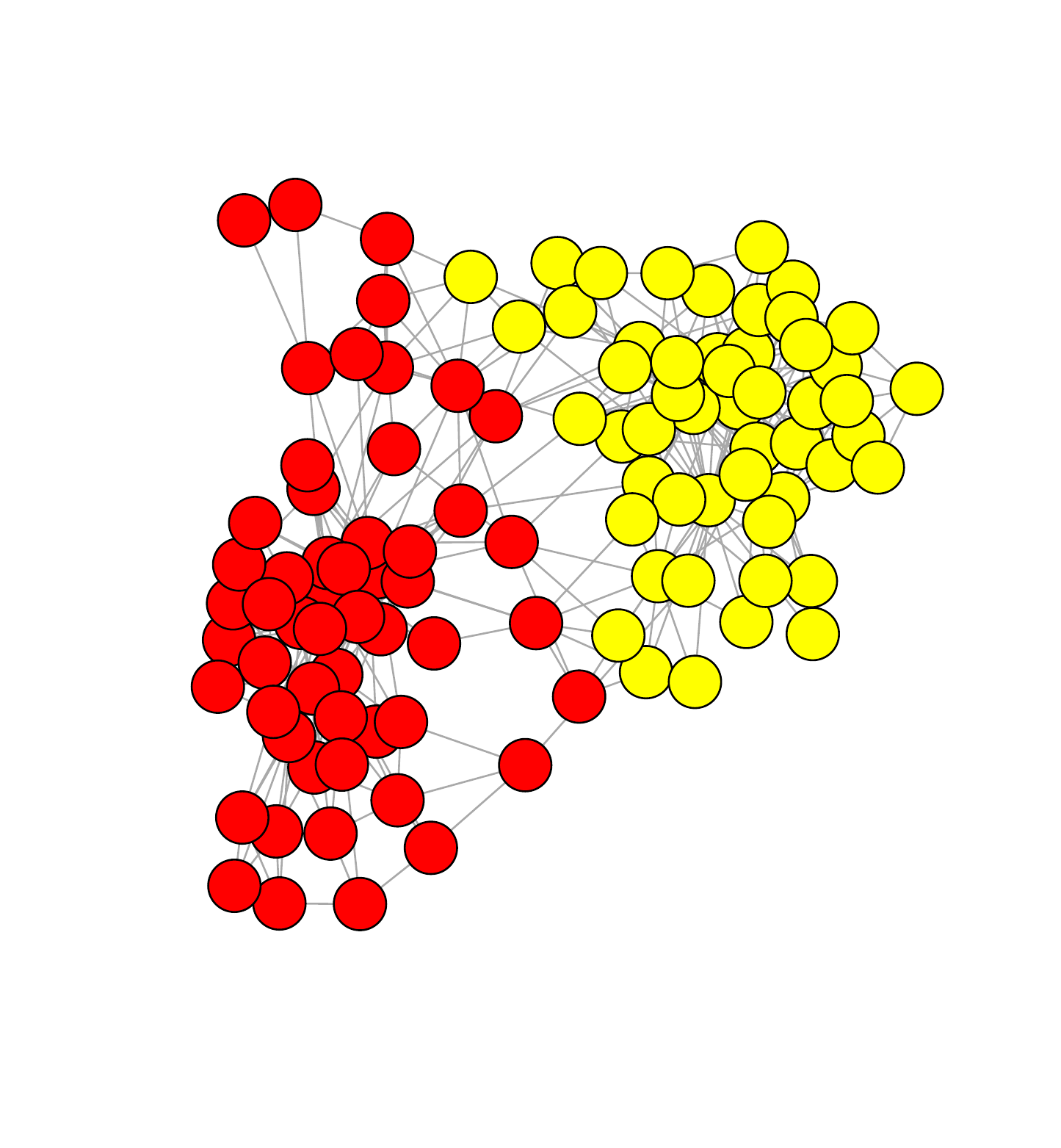}}
\caption{Communities in 105 political books based on (a) manually curated ground truth; (b) penalized likelihood; (c) \texttt{vb}; (d) \texttt{ncv}.}
\label{fig_polbooks}
\end{figure}

Finally we used the political blog network \citep{adamic:2005} as an example of DCSBM. The network consists of blogs on US politics and their web links as edges. Based on whether a blog is ``liberal" or ``conservative", the network is divided into two communities. As is commonly done in the literature, we considered only the largest connected component containing a total of 1222 nodes. In this case, \texttt{plh} selected $K=4$, splitting one of the communities into three as shown in Figure \ref{fig_polblogs}. \texttt{ncv} selected $K=2$. On the other hand, we have also observed in simulations that \texttt{ncv} tends to have a bias toward lower $K$ for DCSBM. 

\begin{figure}[!h]
\centering
\subfloat[]{\includegraphics[width = 2.3in]{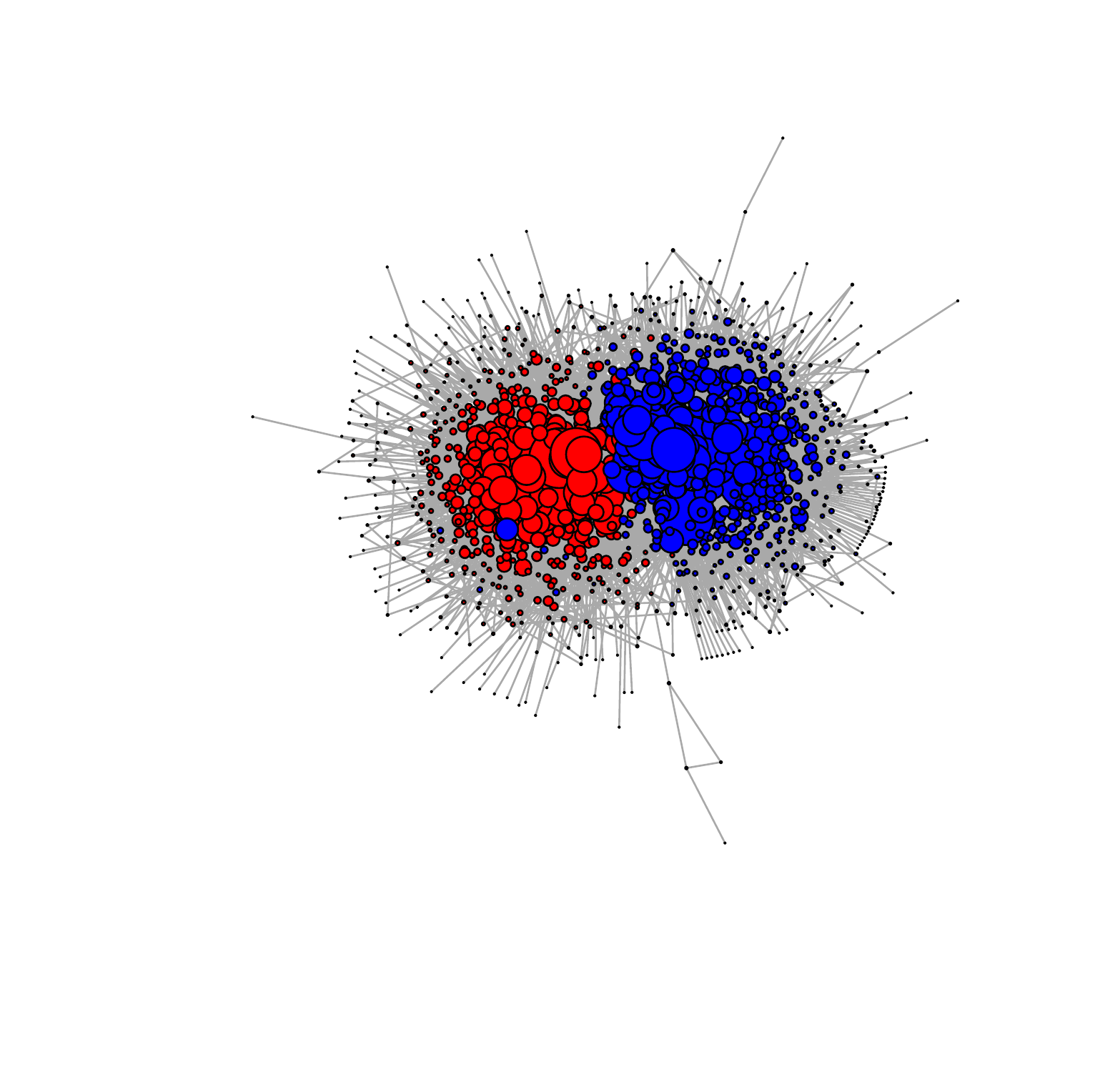}}
\subfloat[]{\includegraphics[width = 2.3in]{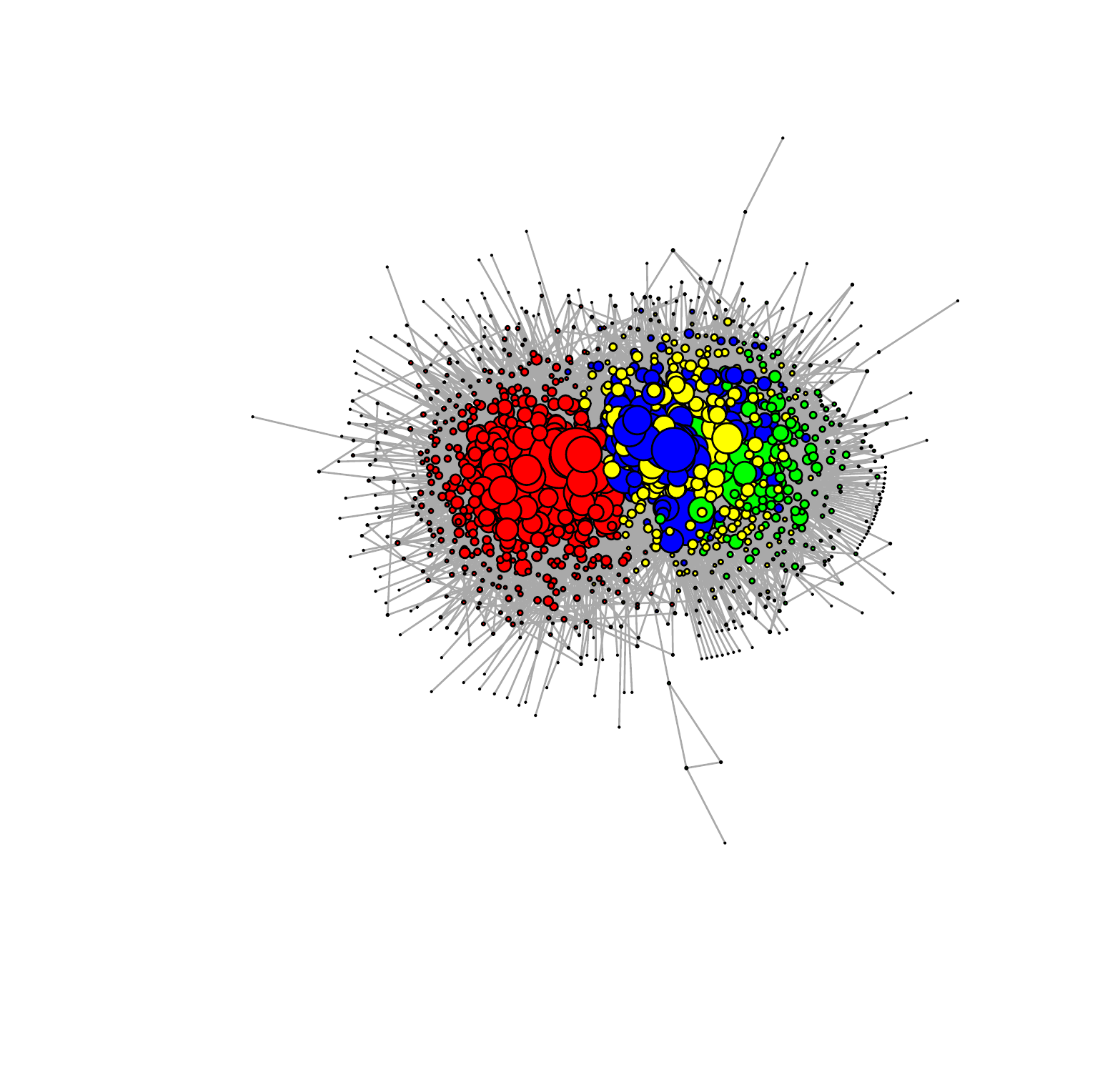}}
\caption{Political blog network (a) manually labeled ground truth; (b) penalized likelihood.}
\label{fig_polblogs}
\end{figure}

\section{Discussion}
In this paper, we have studied the problem of selecting the community number under both regular SBM and DCSBM, allowing the average degree to grow at a polylog rate and the true block number being fixed. We have shown the log likelihood ratio statistic has an asymptotic normal distribution when a smaller model with fewer blocks is specified. In the case of misspecifying a larger model, we have obtained the convergence rate for the statistic. Combining these results we arrive at a likelihood-based model selection criterion that is asymptotically consistent. For finite-sized networks, we have further refined the bound for the statistic in the overfitting case under reasonable assumptions to correct for the possibility of over-penalizing. Our method shows better performance than \texttt{vb} and \texttt{ncv} on simulated data and produces sensible results on a range of real world networks. We also note that \texttt{vb} is only available for regular SBM and \texttt{ncv} tends to be highly varying from run to run due to its use of random partitions. 

There are a number of open problems for future work. (i) It would be interesting to investigate whether the results can be extended to other block model variants, such as  overlapping SBM \citep{Airoldi:2008, ball:2011}. (ii) We have performed our analysis with fixed block number as the number of nodes tends to infinity. However, in practice the number of communities is also likely to grow as a network expands \citep{choi:2012}, especially when we view block models as histogram approximations for more general models \citep{BickelChen:2009, wolfe:2013}. \citet{peixoto:2013} has provided some analysis on the maximum number of blocks detectable for a given SBM graph with fixed labels. In general as more time-course network data become available in biology, social science, and many other domains, incorporating dynamic features of community structures into network modeling will remain an interesting direction to explore. 


\appendix
\section{Proofs of lemmas and theorems}
In this section, we prove all the lemmas and theorems in the main paper. Denote $\mu_n = n^2\rho_n$, the total number of edges $L  =  \sum_{i=1}^{n}\sum_{j=i+1}^{n}A_{i,j}$, and $N(z) = (n_{k,l}(z))_{1\leq k,l \leq K'}$. For two sets of labels $z$ and $y$, $|z -y| = \sum_{i=1}^{n} \I(z_i \neq y_i)$. $\Vert \cdot \Vert_{\infty}$ denotes the maximum norm of a matrix. We abbreviate $R(z, Z) S^* R(z, Z)^T$ as $RS^*R^T(z)$. $C, C_1, \dots$ are constants which might be different at each occurrence. The following concentration inequalities bound the variations in $A$ and will be used throughout the section.

\begin{lemma}
Suppose $z\in[K']^n$ and define $X(z) = O(z)/\mu_n - RS^*R^T(z)$. For $\epsilon\leq 3$,
\begin{equation}
\P\left( \max_{z\in[K']^n} \left\Vert X(z) \right\Vert_{\infty} \geq \epsilon\right) \leq 2(K')^{n+2} \exp\left( -C_1(S^*)\epsilon^2\mu_n \right).
\label{eq_bound1}
\end{equation}
Let $y\in [K']^n$ be a fixed set of labels, then for $\epsilon\leq 3m/n$,  
\begin{align}
& \P\left( \max_{z: |z-y| \leq m} \left\Vert X(z) -  X(y) \right\Vert_{\infty} > \epsilon \right)		\notag\\
\leq & 2\binom{n}{m} (K')^{m+2}\exp\left(- C_2(S^*)\frac{n\epsilon^2 \mu_n}{m} \right).
\label{eq_bound2}
\end{align}
$C_1(S^*)$ and $C_2(S^*)$ are constants depending only on $S^*$.
\label{lem_concentration}
\end{lemma}

\begin{proof}
The proof follows from \cite{BickelChen:2009} with minor modifications for general $K'$-block models and correcting for the zero diagonal in $A$.
\end{proof}

Recall that
\begin{align*}
\gamma_1(x) & = x\log x + (1-x)\log(1-x),	\\
\gamma_2(x) & = x\log x - x.
\end{align*}
Define $F_i(M,t)$, $i=1,2$, as 
\begin{align}
F_i(M, t) = \sum_{k,l=1}^{K'} t_{k,l} \gamma_i\left( \frac{M_{k,l}}{t_{k,l}} \right), 
\end{align}
Then the log of the complete likelihood can be expressed as 
\begin{align}
 & \sup_{\theta\in\Theta_{K'}} \log f(z, A; \theta) 	= n\sum_{k=1}^{K'} \alpha(n_k(z)/n) + \frac{n^2}{2}F_1\left( O(z)/n^2, N(z)/n^2 \right),		\notag\\	
\end{align} 
 where $\alpha(x) = x\log(x)$. Noting the first term is of smaller order compared to the second term, and the conditional expectation of the argument in $\gamma_1$ given $Z$ is $[RH^*R^T(z)]_{k,l}/[R\1\1^T R^T(z) ]_{k,l}$ and $[RS^*R^T(z)]_{k,l} / [R\1\1^T R^T(z) ]_{k,l}$ for $\gamma_2$ (up to a diagonal difference) with fluctuation bounded by Lemma \ref{lem_concentration}, we will focus on analyzing the conditional expectation 
 \begin{align}
 G_1(R(z), S^*) & = \sum_{k,l=1}^{K'} [R\1\1^T R^T(z)]_{k,l} \gamma_1\left( \frac{[RH^*R^T(z)]_{k,l}}{[R\1\1^T R^T(z) ]_{k,l}} \right)	\quad \text{ for } \rho_n = \Omega(1),		\\
 G_2(R(z), H^*) & = \sum_{k,l=1}^{K'} [R\1\1^T R^T(z)]_{k,l} \gamma_2\left( \frac{[RS^*R^T(z)]_{k,l}}{[R\1\1^T R^T(z) ]_{k,l}} \right)	\quad \text{ for } \rho_n \to 0.
 \end{align}
 The following lemma shows in the case of underfitting a $(K-1)$-block model, to maximize $G_i$ over different configurations of $R(z, Z)$ with given $Z$, it suffices to consider the merging scheme described in Section \ref{sec_under} by combining two existing blocks in $Z$. 
 
 
 \begin{lemma}
\label{lem_kl_bound}
Given the true labels $Z$ with block proportions $p=n(Z)/n$, maximizing the function $G_1(R(z), H^*)$ over $R$ achieves its maximum in the label set 
\[
\{z\in[K-1]^n \mid \text{ there exists $\tau$ such that } \tau(z) =U_{a,b}(Z), 1\leq a < b \leq K, \}
\]
where $U_{a,b}$ merges $Z_i$ with labels $a$ and $b$. 

Furthermore, suppose $z_0$ gives the unique maximum (up to permutation $\tau$), for all $R$ such that $R \geq 0, R^T\1= p$, 
\begin{equation}
\left. \frac{\partial G_1((1-\epsilon)R(z_0) + \epsilon R, H^*)}{\partial \epsilon} \right\vert_{\epsilon=0+} < -C <0
\label{eq_neg_deriv}
\end{equation}
for $\rho_n = \Omega(1)$. The same conclusions hold for $G_2(R(z), S^*)$.
\end{lemma}

\begin{proof}
Treating $R$ as a $(K-1)\times K$-dimensional vector, it is easy to check $G_1(\cdot, H^*)$ is a convex function. Furthermore, since $R\geq 0$, $R^T\1 = p$, the domain is part of a convex polyhedron $P_R = \{R\in\R^{K(K-1)} \mid R \geq 0, R^T\1 = p \}$. Therefore the maximum is attained at the vertices of $P_R$, that is $R^{vert}$ such that for every $a$, exactly one $R^{vert}_{k,a}$, $(1\leq k \leq K-1)$ is nonzero. This is equivalent to assigning all $Z_i\in[K]$ with the same label into one group with a new label in $[K-1]$. Let $u: [K] \longrightarrow [K-1]$ be the function specified by $R^{vert}$, then 
\begin{equation}
G_1(R^{vert}, H^*) = \sum_{k,l} \sum_{\substack{a \in u^{-1}(k),\\ b\in u^{-1}(l)}} p_a p_b  \gamma_1 \left( \frac{\sum_{\substack{a \in u^{-1}(k),\\ b\in u^{-1}(l)}} H^*_{a,b} p_a p_b}{\sum_{\substack{a \in u^{-1}(k),\\ b\in u^{-1}(l)}} p_a p_b} \right).
\end{equation}
Note that there exists at least one $l\in [K-1]$ such that $|\{u^{-1}(l)\}| > 1$, and $\{ u^{-1}(k), k\in [K-1] \}$ forms a partition on $[K]$. By strict convexity of $\gamma_1$ and identifiability of $H^*$, to maximize $G_1$ it suffices to consider merging two of the labels in $[K]$ and mapping the other labels to the remaining labels in $[K-1]$ in a one-to-one relationship. 

The second part of the lemma holds since it is easy to see when the maximum is unique, the derivative of the $G_1$ at the optimal vertex is bounded away from 0 in all directions. The same arguments apply to $G_2$. 
\end{proof}

Noting that when $p=\pi^*$, $G_i$ evaluated at $R(U_{a,b}(Z))$ is equal to $D_i$ defined in \eqref{eq_D}, it is easy to see Assumptions \ref{assump_unique} and \ref{assump_iden} guarantees the maximum is unique. We will now prove Lemma \ref{lem_bounds}. 

\begin{proof}[Proof of Lemma \ref{lem_bounds}]
Taking the log of the complete likelihood, 
\begin{align}
 & \sup_{\theta\in\Theta_{K-1}} \log f(z, A; \theta) 	\notag\\
= & n\sum_{k=1}^{K-1} \alpha(n_k(z)/n) + \frac{n^2}{2}F_1\left( O(z)/n^2, N(z)/n^2 \right).		\notag\\	
\end{align}
By concentration of $p_k$, it suffices to consider $\{ \Vert p - \pi^* \Vert_{\infty} <\eta \}$, where $\eta$ is small enough that $Z'$ remains the unique maximizer of $G_1(R(z), H^*)$ and $G_2(R(z), S^*)$, and distribution conditional on $Z$.

Using techniques similar to \cite{bickel:2013}, we prove this by considering $z$ far away from $Z'$ and close to $Z'$ (up to permutation $\tau$). Let 
$\delta_n$ be a sequence converging to 0 slowly. Define
\[
I_{\delta_n} = \{ z\in[K-1]^n : G_1(R(z), H^*) - G_1(R(Z'), H^*) < -\delta_n \}.
\]
First by \eqref{eq_bound1} in Lemma \ref{lem_concentration}, for $\epsilon_n \to 0$ slowly,
\begin{align}
 & \left\vert F_1\left( O(z)/n^2, N(z)/n^2 \right) - G_1(R(z), H^*) \right\vert	\notag\\
\leq & C \cdot \sum_{k,l} \left\vert O_{k,l}(z)/n^2 - (RH^*R^T(z))_{k,l} \right\vert + O(n^{-1})	\notag\\
= & o_P(\epsilon_n)
\end{align}
since $\gamma_1$ is Lipschitz on any interval bounded away from 0 and 1 and $\min H^* = \Omega(1)$. For $z\in I_{\delta_n}$ and $\rho_n = \Omega(1)$, 
\begin{align}
 \sum_{z\in I_{\delta_n}} \sup_{\theta\in\Theta_{K-1}} e^{\log f(z, A; \theta)} & \leq \sup_{\theta\in\Theta_{K-1}} f(Z', A; \theta)  (K-1)^n e^{O(n) + o_P(n^2\epsilon_n) - n^2\delta_n}		\notag\\
& =  \sup_{\theta\in\Theta_{K-1}} f(Z', A; \theta) o_P(1)		
\end{align}
choosing $\delta_n \to 0$ slowly enough such that $\delta_n / \epsilon_n \to \infty$. Similarly for $\rho_n\to 0$, define
\[
J_{\delta_n} = \{ z\in[K-1]^n : G_2(R(z), S^*) - G_2(R(Z'), S^*) < -\delta_n \}.
\]
Note that in this case, for $\epsilon_n \to 0$ slowly, 
\begin{align}
& F_1\left( O(z)/n^2, N(z)/n^2 \right) 		\notag\\
= & 2\log\rho_n L/n^2 + \rho_n F_2 \left( O(z)/\mu_n, N(z)/n^2 \right) + O_P(\rho_n^2)		\notag\\
 = & 2\log\rho_n L/n^2 + \rho_n G_2 \left( R(z), S^* \right) + o_P(\rho_n\epsilon_n) + O_P(\rho_n^2),
 \label{eq_part1_fixed}
\end{align}
by \eqref{eq_bound1} and the fact that $\gamma_2$ is Lipschitz on any interval bounded away from 0 and 1 and $\min S^*_{a,b} >0$. Then for $z\in J_{\delta_n}$,
\begin{align}
 & \sum_{z\in J_{\delta_n}} \sup_{\theta\in\Theta_{K-1}} e^{\log f(z, A; \theta)}		\notag\\
 \leq & \sup_{\theta\in\Theta_{K-1}}f(Z', A; \theta) (K-1)^n e^{O(n)+ O_P(\mu_n\rho_n) + o_P(\mu_n\epsilon_n)-\mu_n\delta_n}	\notag\\
 = & \sup_{\theta\in\Theta_{K-1}}f(Z', A; \theta) o_P(1).
 \label{eq_part1_small}
\end{align}
choosing $\epsilon_n \to 0$, $\delta_n \to 0$ slowly enough. 

For $z\notin J_{\delta_n}$, $| G_2 (R(z), S^*) - G_2(R(Z'), S^*) | \to 0$. Let $\bar{z} = \min_{\tau} |\tau(z)-Z'|$. Since the maximum is unique up to $\tau$, $\Vert R(\bar{z}) - R(Z')\Vert_{\infty} \to0$ and $\left\vert \sum_{k} \alpha(n_k(\bar{z})/n) - \sum_{k} \alpha(n_k(Z')/n) \right\vert \to 0$.  

By \eqref{eq_bound2},  
\begin{align}
 & \P\left( \max_{z\notin\S(Z')} \Vert X(\bar{z}) - X(Z') \Vert_{\infty} > \epsilon|\bar{z}-Z'|/n \right)		\notag\\
\leq & \sum_{m=1}^{n} \P\left( \max_{z: z=\bar{z}, |\bar{z} - Z'| = m} \Vert X(z) - X(Z') \Vert_{\infty} >  \epsilon\frac{m}{n} \right)		\notag\\
\leq & \sum_{m=1}^{n} 2(K-1)^{K-1}n^m (K-1)^{m+2} \exp\left(- C \frac{m\mu_n}{n} \right) 	\to0.
\label{eq_x}
\end{align}
It follows for $|\bar{z} - Z'| = m$, $z\notin J_{\delta_n}$,
\begin{align}
\left\Vert \frac{O(\bar{z})}{\mu_n} - \frac{O(Z')}{\mu_n} \right\Vert _{\infty} & = o_P(1) \frac{|\bar{z} - Z'|}{n}	+ \Vert RS^*R^T(\bar{z}) - RS^*R^T(Z') \Vert_{\infty}		 \notag\\
& \geq \frac{m}{n}(C+o_P(1)).
\end{align}
Observe $\Vert O(Z')/\mu_n - RS^*R(Z') \Vert_{\infty} = o_P(1)$ by Lemma \ref{lem_concentration}, $N(Z')/n^2 = R\1\1^TR^T(Z') + o(1)$ on $\{ \Vert p - \pi^* \Vert_{\infty} <\eta \}$, and $F_2(\cdot, \cdot)$ has continuous derivative in the neighborhood of $(O(Z')/\mu_n, N(Z')/n^2)$. Using \eqref{eq_neg_deriv} in Lemma \ref{lem_kl_bound}, 
\[
\left. \frac{\partial F_2\left( (1-\epsilon)\frac{O(Z')}{\mu_n} + \epsilon M, (1-\epsilon)\frac{N(Z')}{n^2}+ \epsilon t \right)}{\partial \epsilon} \right\vert_{\epsilon=0^+} < -\Omega_P(1) <0 
\]
for $(M, t)$ in the neighborhood of $(O(Z')/\mu_n, N(Z')/n^2)$. Hence
\begin{align}
& F_2\left( O(\bar{z})/\mu_n, N(\bar{z})/n^2 \right) - F_2\left( O(Z')/\mu_n, N(Z')/n^2\right) 	\notag\\
\leq & -\Omega_P(1)\frac{m}{n}.
\label{eq_m_bound}
\end{align}
We have 
\begin{align}
 & \sup_{\theta\in\Theta_{K-1}} \log f(z, A; \theta) - \sup_{\theta\in\Theta_{K-1}} \log f(Z', A; \theta)	\notag\\
\leq & n \left\vert \sum_{k=1}^{K-1} \alpha(n_k(\bar{z})/n) - \alpha(n_k(Z')/n) \right\vert 	\notag\\
 & \quad + n^2\left( F_1(O(\bar{z})/\mu_n, N(z)/n^2) - F_1(O(Z')/\mu_n, N(Z')/n^2) \right)		\notag\\
 \leq &\left( O(n) + o_P(\mu_n) - \Omega_P(\mu_n) \right) \frac{m}{n}	\notag\\
 = & -\Omega_P(\mu_n)\frac{m}{n}	\notag\\
 \label{eq_part2}
\end{align}
using \eqref{eq_part1_fixed} and \eqref{eq_m_bound}. We can conclude 
\begin{align}
 & \sum_{z\notin J_{\delta_n}, z\neq\tau(Z')} \sup_{\theta\in\Theta_{K-1}} e^{\log f(z, A; \theta)}	\notag\\
\leq  & \sup_{\theta\in\Theta_{K-1}} f(Z', A; \theta) \sum_{m=1}^{n} (K-1)^{K-1}n^m (K-1)^m e^{-\Omega_P(\mu_n) m/n}		\notag\\
= & \sup_{\theta\in\Theta_{K-1}} f(Z', A; \theta) o_P(1)
 \label{eq_part2_small}
\end{align}
The bounds \eqref{eq_part1_small} and \eqref{eq_part2_small} yield \eqref{eq_label_bound}. The case for $\rho_n=\Omega(1)$ can be shown in a similar way.
\end{proof}

Now Theorem \ref{thm_underfit} follows by Taylor expansion.
\begin{proof}[Proof of Theorem \ref{thm_underfit}]
First note that 
\begin{align}
L_{K,K-1} & = \log \frac{\sup_{\theta\in\Theta_{K-1}} g(A;\theta)}{g(A;\theta^*)} - \log \frac{\sup_{\theta\in\Theta_K} g(A;\theta)}{g(A;\theta^*)} \notag \\
 & = \sup_{\theta\in\Theta_{K-1}} \log \left[ \frac{g(A;\theta)}{f(Z,A;\theta^*)} \cdot \frac{f(Z,A;\theta^*)}{g(A;\theta^*)} \right] + O_P(1)	\notag\\
 & = \sup_{\theta\in\Theta_{K-1}} \log \frac{g(A;\theta)}{f(Z,A;\theta^*)} + O_P(1)
\end{align}
by a consequence of Theorem 1 and Lemma 3 in \cite{bickel:2013}. Noting that $\sup_{\theta\in\Theta_{K-1}} f(Z', A; \theta)$ is uniquely maximized at (omitting the argument $Z$)
 \begin{align}
 & \hat{\pi}_a = \frac{n_a}{n} = \pi^*_{a} + O_P(n^{-1/2}) \text{ for } 1\leq a\leq K-2,\quad \hat{\pi}_{K-1} = \frac{n_{K-1}+n_K}{n} \notag\\
  & \hat{H}_{a,b} = \frac{O_{a,b}}{n_{a,b}} = H^*_{a,b} + O_P(\sqrt{\rho_n}n^{-1}) \text{ for } 1\leq a \leq b \leq K-2, \notag \\
  & \hat{H}_{a,K-1} = \frac{O_{a,K-1}+O_{a,K}}{n_{a,K-1}+n_{a,K}} = H'_{a,K-1} + O_P(\rho_n n^{-1/2})\text{ for } 1\leq a \leq K-2, \notag\\
  & \hat{H}_{K-1,K-1} = \frac{\sum_{a=K-1}^K \sum_{b=a}^{K} O_{a,b}}{\sum_{a=K-1}^K \sum_{b=a}^{K} n_{a,b}} = H'_{K-1,K-1} + O_P(\rho_n n^{-1/2}),
\end{align}
and Assumption \ref{assump_iden} the merged $S'$ is identifiable, we have \[
\frac{\sup_{\theta\in\Theta_{K-1}}\sum_{z\in\S(Z')}f(z,A;\theta)}{\sup_{\theta\in\Theta_{K-1}}f(Z',A;\theta)} = 1+o_P(1).
\]
Combined with Lemma \ref{lem_bounds} 
\begin{align}
 & \sup_{\theta\in\Theta_{K-1}} \log\frac{g(A;\theta)}{f(Z,A;\theta^*)} \notag\\
 = & \sup_{\theta\in\Theta_{K-1}} \log\frac{f(Z',A;\theta)}{f(Z,A;\theta^*)} + o_P(1). \notag \\
\end{align}

We will check the expansion for the case $\rho_n\to0$; the case $\rho_n=\Omega(1)$ can be shown in the same way.  
\begin{align}
& n^{-3/2}\rho_n^{-1} \sup_{\theta\in\Omega_{K-1}} \log\frac{g(A;\theta)}{f(Z,A;\theta^*)} \notag\\
= & n^{-3/2}\rho_n^{-1} \sup_{\theta\in\Omega_{K-1}} \log\frac{f(Z', A;\theta)}{f(Z, A;\theta^*)} + o_P(1)	\notag\\
= & n^{-3/2}\rho_n^{-1} \left\{ n \sum_{a=1}^{K-1} \alpha(\hat{\pi}_{a})+  \frac{1}{2}\sum_{a=1}^{K-1}\sum_{b=1}^{K-1} \left( O_{a,b}\log\frac{\hat{H}_{a,b}}{1-\hat{H}_{a,b}} + n_{a,b} \log(1-\hat{H}_{a,b})\right)  \right.	 \notag \\
  & \left.\quad -\sum_{a=1}^{K} n_a \log \pi_a^* - \frac{1}{2}\sum_{a=1}^{K}\sum_{b=1}^{K} \left( O_{a,b}\log\frac{H^*_{a,b}}{1-H^*_{a,b}} + n_{a,b} \log(1-H^*_{a,b})\right)  \right\} + o_P(1)	\notag\\
= & n^{-1/2}\rho_n^{-1} \left[ \alpha(\pi^*_{K-1} + \pi^*_{K}) - \alpha(\pi^*_{K-1}) -\alpha(\pi^*_{K}) \right]		\notag\\
 & \quad n^{-3/2}\rho_n^{-1} \frac{1}{2}\sum_{(a,b)\in\mathcal{I}}  \left( O_{a,b} \log \frac{H'_{u(a), u(b)}(1-H^*_{a,b})}{(1-H'_{u(a),u(b)})H^*_{a,b}} + n_{a,b} \log \frac{1-H'_{u(a),u(b)}}{1-H^*_{a,b}} \right) + o_P(1),		 \notag\\
 	 \notag\\
\end{align}
where $\mathcal{I}$ be the set of indices affected by the merge, 
\[
\mathcal{I} = \{ (a,b)\in[K]^2 \mid K-1 \leq a \leq K \text{ or } K-1 \leq b \leq K \}.
\]
It is easy to see the expectation of this term is $\rho_n^{-1} \sqrt{n} \mu_2$, we have
\begin{align}
 & n^{-3/2}\rho_n^{-1} \sup_{\theta\in\Omega_{K-1}} \log\frac{g(A;\theta)}{f(Z,A;\theta^*)} - \sqrt{n}\rho_n^{-1} \mu_2 \notag\\
= & \frac{1}{2n^{3/2} \rho_n}\sum_{(a,b)\in\mathcal{I}} \left[ (O_{a,b} - n^2H^*_{a,b}\pi_a^*\pi_b^*) \log \frac{H'_{u(a),u(b)}(1-H^*_{a,b})}{(1-H'_{u(a),u(b)})H^*_{a,b}} 	\right.	\notag\\
& \left. + (n_{a,b} - n^2\pi_a^*\pi_b^*) \log \frac{1-H'_{u(a),u(b)}}{1-H^*_{a,b}} \right]  + o_P(1)	 \notag \\
 = & \frac{1}{2n^{3/2} \rho_n}\sum_{(a,b)\in\mathcal{I}} (n_{a,b} - n^2\pi_a^*\pi_b^*) d_1(a,b) + o_P(1),	\notag\\
 =  & \frac{\sqrt{n}}{2}\sum_{(a,b)\in\mathcal{I}} (n_{a,b}/n^2 - \pi_a^*\pi_b^*) d_2(a,b) + o_P(1)
\end{align}
where $d_i(a,b)$ is defined in \eqref{eq_def_d}. \eqref{eq_asy_gaussian} follows by the delta method.

\end{proof}

\begin{proof}[Proof of Lemma \ref{lem_overfit_split}]
The proof follows using arguments similar to Lemma \ref{lem_bounds}. Note that in this case $G_1(R(z), H^*)$ is maximized at any $z \in\sV_{K^+}$ with the value $\sum_{a,b} p_a p_b \gamma_1 (H^*_{a,b})$ (or  $\sum_{a,b} p_a p_b \gamma_2 (S^*_{a,b})$ for $G_2(R(z), S^*)$). 

It suffices to discuss the case $\rho_n \to 0$. Denote the optimal $G^* := \sum_{a,b} p_a p_b \gamma_2 (S^*_{a,b})$, define similarly to Lemma \ref{lem_bounds}
\[
J_{\delta_n} = \{ z\in[K^+]^n : G_2(R(z), S^*) - G^* < -\delta_n \}
\]
for $\delta_n \to 0$ slowly enough. It is easy to see \[
\sum_{z\in J_{\delta_n}} \sup_{\theta\in\Theta_{K^+}} f(z, A; \theta) \leq \sup_{\theta\in\Theta_{K^+}} f(z_0, A;\theta) o_P(1)
\]
for any $z_0 \in \sV_{K^+}$. 

Next note that treating $R(z)$ as a vector, $\{R(z) \mid z\in\sV_{K^+}\}$ is a subset of the union of some of the $K^+-K$ faces of the polyhedron $P_R$. For every $z \notin J_{\delta_n}, z\notin \sV_{K^+}$, let $z_{\perp}$ be such that $R(z_{\perp}) := \min_{R(z_0): z_0\in\sV_{K^+}} \Vert R(z)-R(z_0) \Vert_2$. $R(z)-R(z_{\perp})$ is perpendicular to the corresponding $K^*-K$ face. Furthermore, this orthogonality implies the directional derivative of $G_2(\cdot, S^*)$ along the direction of $R(z)-R(z_{\perp})$ is bounded away from 0. That is
\[
\left. \frac{\partial G_2\left( (1-\epsilon)R(z_{\perp}) + \epsilon R(z), S^* \right) }{\partial \epsilon} \right\vert_{\epsilon = 0^+} < -C
\]
for some universal positive constant $C$. Similar to \eqref{eq_part2},
\begin{align*}
\sup_{\theta\in\Theta_{K^+}} \log f(z, A; \theta) - \sup_{\theta\in\Theta_{K^+}} \log f(z_{\perp}, A; \theta) & \leq -\Omega_P(\mu_n) \frac{m}{n}		\\
\sup_{\theta\in\Theta_{K^+}} f(z, A; \theta) & \leq e^{-\Omega_P(\mu_n) \frac{m}{n}} \sup_{\theta\in\Theta_{K^+}} f(z_{\perp}, A; \theta)
\end{align*}
where $|z - z_{\perp}| = m$. We have
\begin{align*}
& \sum_{z \notin J_{\delta_n}, z\notin \sV_{K^+}} \sup_{\theta\in\Theta_{K^+}} f(z, A; \theta)	\\
\leq & \sum_{z \in \sV_{K^+}}  \sup_{\theta\in\Theta_{K^+}} f(z, A; \theta) \sum_{m=1}^{n} (K-1)^m n^m e^{-\Omega_P(\mu_n) \frac{m}{n}}		\\
= & o_P(1)\sum_{z \in \sV_{K^+}}  \sup_{\theta\in\Theta_{K^+}} f(z, A; \theta).
\end{align*}
Hence the claim follows. 
\end{proof}

\begin{proof}[Proof of Theorem \ref{thm_overfit}]
First note 
\[
L_{K,K^+} = \log \frac{\sup_{\theta\in\Theta_{K^+}} g(A;\theta)}{f(Z, A; \theta^*)} +O_P(1),
\]
where 
\begin{align}
\log \frac{\sup_{\theta\in\Theta_{K^+}} g(A;\theta)}{f(Z, A; \theta^*)} &  \geq \log \frac{\sup_{\theta\in\Theta_{K^+}} f(Z, A;\theta)}{f(Z, A; \theta^*)} 	\notag\\
& =  O_P(1).	
\end{align}

Let $D(\cdot)$ be a diagonal matrix, upper bounding by the maximum,
\begin{align}
& \log \frac{\sup_{\theta\in\Theta_{K^+}} g(A;\theta)}{f(Z, A; \theta^*)}		\notag\\
\leq & \max_{z}\sup_{\theta\in\Theta_{K^+}} \log \frac{f(z,A;\theta)}{f(Z, A; \theta^*)} +	n\log K^+ 		\notag\\
= & \max_{z} \frac{n^2}{2} \left\{ F_1\left( O(z)/n^2, N(z)/n^2\right) - F_1 \left(D(p)H^*D(p), pp^T\right) \right\} + O_P(n)		\notag\\
\leq & \max_{z} \frac{n^2}{2} \left\vert F_1\left( O(z)/n^2, N(z)/n^2 \right)  - F_1(RH^*R^T(z), R\1\1^TR^T(z)) \right\vert		\notag\\
 & \quad  +  \max_{z} \frac{n^2}{2} \left[ F_1(RH^*R^T(z), R\1\1^TR^T(z)) - F_1 \left(D(p)H^*D(p), pp^T\right) \right]  +O_P(n)		\notag\\
\leq & C\mu_n \max_{z}\left\Vert \frac{O(z)}{\mu_n} - RS^*R^T\right\Vert_{\infty} + O_P(n)	\notag\\
= & O_P(n^{3/2}\rho_n^{1/2})
\end{align}
using \eqref{eq_bound1} in Lemma \ref{lem_concentration}, and the fact that 
\[
\max_{z\in[K^+]^{n}} F_1(RH^*R^T(z), R\1\1^TR^T(z)) = F_1 \left(D(p)H^*D(p), pp^T\right).
\]
\end{proof}

Next we prove Theorem \ref{thm_overfit_refine}.
\begin{proof}[Proof of Theorem \ref{thm_overfit_refine}]
To upper bound $L_{K, K^+}$, by Lemma \ref{lem_overfit_split}, it suffices to consider 
\begin{align*}
& \max_{z\in\sV_{K^+}} \sup_{\theta\in\Theta_{K^+}} \log f(z,A;\theta) - \sup_{\theta\in\Theta_{K}} \log g(A; \theta)	+ O(n)	\\
= & \max_{z\in\sV_{K^+}} \sup_{\theta\in\Theta_{K^+}} \log f(z,A;\theta) - \log f(Z, A; \theta^*) + O(n).
\end{align*}
It follows from the definition of $\sV_{K^+}$ there exists a surjective function $h: [K^+] \to [K]$ describing the block assignments in $R(z, Z)$. We have
\begin{align}
& \max_{z\in\sV_{K^+}} \sup_{\theta\in\Theta_{K^+}} \log f(z,A;\theta) - \log f(Z, A; \theta^*) + O(n)		\notag\\
= & \max_{z\in\sV_{K^+}}  n \sum_{k=1}^{K^+} \alpha \left( n_k(z)/n \right) - n \sum_{a=1}^{K} \alpha\left( \sum_{k\in h^{-1}(a)} n_k(z)/n \right)	\notag\\
& + \frac{1}{2}\sum_{k=1}^{K^+}\sum_{l=1}^{K^+}  \left( O_{k,l} \log\frac{\hat{H}_{k,l}}{H^*_{h(k), h(l)}}  + (n_{k,l} - O_{k,l}) \log \frac{1-\hat{H}_{k,l}}{1-H^*_{h(k), h(l)}} \right) + O_P(n),		
\label{eq_overfit_expansion}
\end{align}
where $\hat{H}_{k,l} = O_{k,l}(z) / n_{k,l}(z)$. The first part of the expression is nonpositive since $\alpha$ is superadditive.  

Taylor expanding \eqref{eq_overfit_expansion} and using the fact that $\hat{H}_{k,l} - H^*_{h(k), h(l)} = O_P(n^{-1/2}\rho_n^{1/2})$ for $z\in\sN_{K^+}$ uniformly, \eqref{eq_overfit_expansion} is upper bounded by $O_P(n)$.
\end{proof}

\section*{Acknowledgments}

The authors thank Haiyan Huang for helpful discussions and the reviewers for their valuable comments and suggestions. This research was funded in part by NSF FRG Grant DMS-1160319.

\begin{supplement}
\stitle{Supplement to ``Likelihood-based model selection for stochastic block models''} \slink[url]{http://???} 
\sdescription{A proof sketch of how the main results in the paper can be extended to the DCSBM described in Section \ref{subsec_dcsbm} is provided in the supplement.}
\end{supplement}

\bibliographystyle{imsart-nameyear}
\bibliography{ref}



\section*{Supplementary material}

# SUPPLEMENT TO "LIKELIHOOD-BASED MODEL SELECTION FOR STOCHASTIC BLOCK MODELS"

By Y.X. Rachel Wang[*] and Peter J. Bickel[†]

*Department of Statistics, Stanford University*[*]
*Department of Statistics, University of California, Berkeley*[†]

We provide a sketch of a proof showing the main results hold for the DCSBM described in Section 2.5. For labels $z \in [K']^n$, define counts

$$O_a(z) = \sum_{i,j} \mathbb{I}(z_i = a) A_{i,j}.$$

We divide the sketch into three parts.

i) *Lemma 2.3 still holds.* First note that the likelihood function $f(z, A; \theta)$ is given by (up to a constant)

$$f(z, A; \theta) = \prod_k \pi_k^{n_k(z)} \exp\left( \frac{1}{2} \sum_{k,l} (O_{k,l}(z) \log H_{k,l} - n_k(z) n_l(z) H_{k,l}) \right)$$

$$\times \prod_k \int_{\sum_{i:z_i=k} \omega_i = n_k(z)} \frac{1}{(n_k(z))^{n_k(z)} B(\mathbf{1}_{n_k(z)})} \prod_{i:z_i=k} \omega_i^{\sum_j A_{i,j}} d\omega$$

$$= \prod_k \pi_k^{n_k(z)} \exp\left( \frac{1}{2} \sum_{k,l} (O_{k,l}(z) \log H_{k,l} - n_k(z) n_l(z) H_{k,l}) \right)$$

(0.1) $$\times \prod_k n_k(z)^{O_k(z)} \frac{B((\nu_i + 1)_{i:z_i=k})}{B(\mathbf{1}_{n_k(z)})},$$

where $\nu_i$ is the degree of node $i$ and $B(\cdot)$ is the beta function. It follows then up to a constant,

$$\sup_{\theta \in \Theta_{K'}} \log f(z, A; \theta)$$

$$= n \sum_k \alpha(n_k(z)/n) + \frac{1}{2} \sum_{k,l} \left( O_{k,l}(z) \log \frac{O_{k,l(z)}}{n_k(z) n_l(z)} - O_{k,l}(z) \right)$$

$$+ \sum_k O_k(z) \log n_k(z) + \sum_k \log((n_k(z) - 1)!)$$

$$- \log((O_k(z) + n_k(z) - 1)!)$$

1

$$= n \sum_k \alpha(n_k(z)/n) + \frac{n^2}{2} F_2\left(O(z)/n^2, N(z)/n\right)$$

(0.2)
$$- \sum_k (O_k(z) + n_k(z)) \log\left(\frac{O_k(z)}{n_k(z)n} + \frac{1}{n}\right) + O(\log n) + \text{constant}.$$

Conditional on the true labels and degree parameters $(Z, \omega)$, define a new $K' \times K$ confusion matrix $R^c$ whose $(k, a)$-th entry is

$$R^c_{k,a} = n^{-1} \sum_i \omega_i \mathbb{I}(z_i = k, Z_i = a).$$

Using appropriate Chernoff bounds, we can show $O_{k,l}(z)/\mu_n$ concentrates around the expectation $(R^c S^*(R^c)^T)_{k,l}$ and a bound similar to (A.1) holds, that is, there exists a sequence $\epsilon_n \to 0$ such that

$$\max_{z \in [K']^n} \|X^c(z)\|_\infty = o_P(\epsilon_n),$$

where $X^c(z) = O(z)/\mu_n - R^c S^*(R^c)^T(z)$. The same uniform bound holds for $\max_{z \in [K']^n, k \in [K']} \left|\frac{O_k(z)}{\mu_n} - (R^c S^*(R^c)^T(z)\mathbf{1})_k\right|$. Therefore the key term is still the one involving conditional expectations

$$G_3(R^c(z), S^*) = G_2(R^c(z), S^*) - \sum_k (R^c(z)\mathbf{1})_k \alpha\left(\frac{[R^c S^*(R^c)^T(z)\mathbf{1}]_k}{[R^c(z)\mathbf{1}]_k}\right)$$

recalling $\alpha(x) = x \log x$. Noting that $G_3(\cdot, S^*)$ is convex and its domain $\{R^c \in \mathbb{R}^{K' \times K} \mid R^c \geq 0, (R^c)^T \mathbf{1} = p\}$ is again part of a convex polyhedron, the arguments in Lemma A.2 apply and we can conclude that

1. When $K' = K$, $R^c(Z) = diag(p)$ gives the unique maximum.
2. When $K = K - 1$, with an assumption similar to Assumption 2.1 a unique maximum is achieved at the configuration merging the correct pair of blocks.
3. (A.7) holds for $G_3(R^c, S^*)$.

Therefore the first part of the proof of Lemma 2.3 still holds for $z$ far away from $Z'$ (or $Z$ for a $K$-block model).

For $|z - Z'| = m$, the convergence rates have to be adjusted for the degree parameters $\omega$. Let $s(z) = \sum_{i:z_i \neq Z'_i} \omega_i$ and $s_2(z) = \sum_{i:z_i \neq Z'_i} \omega_i^2$. First we note that for $|z - y| \leq m$, the following concentration holds using Chernoff bounds,

$$\mathbb{P}\left(\max_{z:|z-y|=m} \|X(z) - X(y)\|_\infty > \epsilon\right)$$

$$(0.3) \qquad \leq \sum_{z:|z-y|=m} (K')^2 \exp\left(-\frac{C\epsilon^2 \mu_n^2}{ns(z)\rho_n + \rho_n^2 s_2(z) \sum_i \omega_i^2}\right).$$

Next noting that $|R^c(\bar{z}) - R^c(Z')|_1 = Cs(\bar{z})/n$, we have

$$\mathbb{P}\left(\max_{z \notin \mathcal{S}(Z')} \frac{\|X(\bar{z}) - X(Z')\|_\infty}{|R^c(\bar{z}) - R^c(Z')|_1} > \epsilon\right)$$

$$\leq \sum_{m>0} \sum_{z:|\bar{z}-Z'|=m} (K')^2 \exp\left(-\frac{Cs^2(\bar{z})\mu_n^2/n^2}{ns(\bar{z})\rho_n + \rho_n^2 s_2(\bar{z}) \sum_i \omega_i^2}\right)$$

$$\leq \sum_{m>0} \sum_{z:|\bar{z}-Z'|=m} (K')^2 \exp\left(-C \min\left\{\frac{s(\bar{z})\mu_n}{n}, \frac{n^2 s^2(\bar{z})}{s_2(\bar{z}) \sum_i \omega_i^2}\right\}\right) \to 0$$

(0.4)

for $n^{1/2} \rho_n / \log n \to \infty$, since $\sum_i \omega_i^2 = O_P(n)$ and $s^2(\bar{z})/s_2(\bar{z}) = \Omega_P(mn^{-1/2}(\log n)^{-1})$. Therefore

$$\left\|\frac{O(\bar{z})}{\mu_n} - \frac{O(Z')}{\mu_n}\right\|_\infty$$
$$= o_P(1)|R^c(\bar{z}) - R^c(Z')|_1 + \|R^c S^*(R^c)^T(\bar{z}) - R^c S^*(R^c)^T(Z')\|_\infty$$
$$(0.5) \qquad \geq \|R^c(\bar{z}) - R^c(Z')\|_1 (C + o_P(1)) = \Omega_P(m/n^{3/2}).$$

The rest of the proof follows and Lemma 2.3 holds for both $K-1$ and $K$-block models.

ii) *Theorem 2.4 holds with different mean and variance terms.* The maximum likelihood estimates have the same convergence rates as in the case of standard SBM, similar expansions show $L_{K,K-1}$ scaled the same way is asymptotically normal. Furthermore, as discussed in Section 2.2, the result is generalizable to $L_{K,K^-}$ for $K^- < K$ with similar assumptions and $L_{K,K^-} \leq \Omega_P(n^2 \rho_n)$.

iii) *Lemma 2.6 follows* by similar arguments to (A.31) and (A.32). Theorems 2.7 and 2.10 hold by similar expansions.

Y.X.R. Wang
Department of Statistics
Sequoia Hall
390 Serra Mall
Stanford University
Stanford, CA 94305
E-mail: ryxwang@stanford.edu

P.J. Bickel
Department of Statistics
University of California, Berkeley
367 Evans Hall
Berkeley, California, 94720
E-mail: bickel@stat.berkeley.edu


\end{document}